\documentclass[review]{elsarticle}
\usepackage{graphicx}
\usepackage{lipsum}
\usepackage{algorithmic}
  \usepackage{geometry}
\usepackage{epstopdf}
 \usepackage{mathrsfs}
  \usepackage{booktabs,longtable}
   \usepackage{indentfirst}
\usepackage{epsfig}
\usepackage{lineno,hyperref}
\usepackage{bookmark}
\usepackage{multirow}
 \usepackage{amsfonts}
\usepackage{multirow}
 \usepackage{algorithm}
 \usepackage[figuresright]{rotating}
 \usepackage{amscd}
\usepackage{mathrsfs}
\usepackage{subfigure}
\usepackage{amstext}
\usepackage{amssymb}
\usepackage{fancyhdr}
\usepackage{amsmath}
\usepackage{color}
\usepackage{cleveref}
\modulolinenumbers[5]
\biboptions{sort&compress}
\newtheorem{theorem}{Theorem}
\newtheorem{lemma}{Lemma}
\newdefinition{remark}{Remark}
\newproof{proof}{Proof}
\newtheorem{definition}{Definition}

\journal{Journal of \LaTeX\ Templates}

\begin{document}

\begin{frontmatter}

\title{Greedy randomized sampling nonlinear Kaczmarz methods \tnoteref{mytitlenote}}
\tnotetext[mytitlenote]{The work is supported by the National Natural Science Foundation of China (No. 11671060) and the Natural Science Foundation of Chongqing, China (No. cstc2019jcyj-msxmX0267)\\ *Corresponding author\\ Email address: lihy.hy@gmail.com or hyli@cqu.edu.cn (Hanyu Li)}

\author{Yanjun Zhang}
\author{Hanyu Li*}

\author {Ling Tang}

\address{College of Mathematics and Statistics, Chongqing University, Chongqing 401331, P.R. China}
\begin{abstract}
The nonlinear Kaczmarz method was recently proposed to solve the system of nonlinear equations. In this paper, we first discuss two greedy selection rules, i.e.,  the maximum residual and maximum distance rules, for the nonlinear Kaczmarz iteration. Then, based on them, two kinds of greedy randomized sampling methods are presented. 
Further, we also devise four corresponding greedy randomized block methods, i.e., the multiple
samples-based methods. 
The linear convergence in expectation of all the proposed methods is proved. Numerical results show that, in some applications including brown almost linear function and generalized linear model, the greedy selection rules give faster convergence rates than the random ones, and the block methods outperform the single sample-based ones.
\end{abstract}


\end{frontmatter}

\linenumbers

\section{Introduction}
\label{sec;introduction}
We consider finding roots of system of nonlinear equations
\begin{align}
f(x)=0,  \label{Sec11}
\end{align}
where $x\in\mathbb{R}^{n}$ is an unknown variable and $f:\mathbb{R}^{n}\rightarrow\mathbb{R}^{m}$ is a continuously differentiable vector-valued function, which can be written as $f(x)=[f_{1}(x), \cdots, f_{m}(x)]^T\in\mathbb{R}^{m}$ with $f_i:\mathbb{R}^{n}\rightarrow\mathbb{R}~(i=1, \cdots, m)$ being a real-valued function. Throughout this paper, we assume that there exists a solution $x_\star$ satisfying the nonlinear equations (\ref{Sec11}), i.e., $f(x_\star)=0$.

Solving system of nonlinear equations is an ubiquitous problem that 
plays crucial role in a wide range of applied fields such as machine learning \cite{bjorck1996numerical, chen2019homotopy}, differential equations \cite{dennis1996numerical, hao2014bootstrapping}, integral equations \cite{atkinson1992survey, hao2018equation} and optimization problems \cite{hao2018homotopy, kelley1999iterative}. Meanwhile, solving most convex optimization problems and non-convex optimization problems can also be converted into finding a good stationary point $x$ such that its gradient equals the zero vector. On the other hand, with the advent of the big data era, large-scale nonlinear problems are widely present in various fields. Therefore, 
it is very important to establish efficient methods for system of nonlinear equations both in practice and theory.

A classical algorithm for solving nonlinear problems (\ref{Sec11}) is the Newton-Raphson method \cite{ortega2000iterative} given by
\begin{align}
x_{k+1}=x_{k}-(f^{\prime}(x_k))^{\dagger}f(x_k),  \notag
\end{align}
where $f^{\prime}(x)=[\nabla f_1(x), \cdots, \nabla f_m(x)]^T\in\mathbb{R}^{m\times n}$ is the Jacobian matrix of $f$ at $x$ and $\nabla f_i(x)^T$ is its $i$th row, and $(f^{\prime}(x))^{\dagger}$ is the Moore-Penrose pseudoinverse of $f^{\prime}(x)$. Noting that it involves calculating and storing the entire Jacobian matrix, as well as the computation of its inverse or Moore-Penrose pseudoinverse, the Newton-Raphson method is hard to be applied in large dimensional setting. Recently, the nonlinear Kaczmarz method \cite{yuan2022sketched, zeng2020successive, wang2022nonlinear}, which is also called Polyak step size method in \cite{polyak1987introduction, prazeres2021stochastic, loizou2021stochastic}, is proposed to solve nonlinear problems (\ref{Sec11}). Actually, the 
method can also be seen as a subsampled Newton Raphson method \cite{yuan2022sketched}. It consists of sequential orthogonal projections towards the linearization of one constraint set associated with a single nonlinear equation. More specifically, at each iteration, it first samples a single index $i$ from $\{1, \cdots, m\}$ and linearize $f_i(x)$ around a given $x_k\in\mathbb{R}^n$, and set the linearization of $f_i(x)$ to zero, that is,
\begin{align}
f_i(x_k)+\nabla f_i(x_k)^T(x-x_k)=0. \notag
\end{align}
Then, the next iteration $x_{k+1}$ can be obtained in terms of orthogonal projection of the vector $x_{k}$ onto the above constraint, that is,
\begin{align}
x_{k+1}=\mathop{\text{argmin}}\limits_{x\in \mathbb{R}^{n}}\|x-x_k\|_2^2,\quad s.t. \quad f_i(x_k)+\nabla f_i(x_k)^T(x-x_k)=0.\label{Sec1.3}
\end{align}
Finally, we can get the update formula of the nonlinear Kaczmarz method from (\ref{Sec1.3}) as follows:
\begin{align}
x_{k+1}=x_{k}-\frac{f_i(x_k)}{\|\nabla f_i(x_k)\|^2_2}\nabla f_i(x_k).\label{Sec1.4}
\end{align}
From (\ref{Sec1.4}), we know that in each iteration the nonlinear Kaczmarz method only needs to compute one row of the Jacobian matrix instead of the entire matrix, which greatly reduces the amount of calculation and storage.

When $f$ takes the form $f(x)=Ax-b$, (\ref{Sec11}) reduces to a system of linear equations $Ax=b$, and the nonlinear Kaczmarz method  (\ref{Sec1.4}) reduces to the classical Kaczmarz method \cite{kaczmarz1}:
\begin{align}
x_{k+1}=x_{k}-\frac{A^{(i)}x_k-b^{(i)}}{\|A^{(i)}\|^2_2}(A^{(i)})^T,    \notag
\end{align}
where $A^{(i)}$ denotes the $i$th row of the matrix $A$ and $b^{(i)}$ represents the $i$th entry of the vector $b$. In 2009, Strohmer
and Vershynin \cite{strohmer2009randomized} showed that picking a row index with probability proportional to the square of the Euclidean norm of the row, the randomized Kaczmarz (RK) method converges linearly in expectation. Subsequently, to further improve its convergence rate, two greedy rules including maximum residual \cite{griebel2012greedy, nutini2016convergence}, which is also known as the Motzkin method, and maximum distance \cite{nutini2016convergence, Gao2019} are widely used in the literature \cite{Eldar2011, Bai2018, zhang2021count, zhang2020greedy}. In particular, combining the ideas of the RK and Motzkin methods, De Loera et al. \cite{de2017sampling} and Haddock et al. \cite{haddock2021greed} presented the sampling Kaczmarz-Motzkin (SKM) method, which inherits the advantages of the previous two methods.

Considering the good performance and wide applications of greedy rules in classical Kaczmarz type methods, we discuss two greedy rules, i.e., the maximum residual and maximum distance rules, in detail for nonlinear Kaczmarz iteration in this paper, 
which can be regarded as a generalization from linear to nonlinear. Further, inspired by the algorithmic framework in the SKM method, we propose two different adaptive greedy randomized methods for solving the nonlinear equations (\ref{Sec11}). We prove that the two new methods converge linearly in expectation with convergence factors that are strictly smaller than that of the methods presented in \cite{wang2022nonlinear}. Furthermore, four greedy block versions are also proposed to accelerate our new methods.

The rest of this paper is organized as follows. \Cref{sec:preliminaries} provides some preliminaries.  In \Cref{sec: Single-methods}, we first discuss two greedy rules for nonlinear Kaczmarz iteration and then propose two greedy sampling randomized nonlinear methods. The relevant convergence analysis are presented in \Cref{sec:theory}. \Cref{sec: mutiple-methods} gives four block greedy sampling nonlinear methods and their corresponding convergence theorems are provided in \Cref{sec:theory-mutiple}. Numerical experiments in \Cref{sec:experiments} are devoted to illustrate our methods for solving different problems including brown almost linear function and generalized linear model. Finally, the concluding remarks of the whole paper are given in \Cref{sec:conclusions}.

\section{Preliminaries}
\label{sec:preliminaries}

\subsection{Notation}
\label{subsec:notation}
 For a matrix $A=(A_{(i, j)})\in \mathbb{R}^{m\times n}$, $\sigma_{\max}(A)$, $\|A\|_2$, $\|A\|_F$ and $A_{\tau}$ denote its largest singular value, spectral norm, Frobenius norm, and the restriction onto the row indices in the set $\tau$, respectively. If $A$ is a square matrix, i.e., $m=n$,  $\lambda(A)$ stands for an eigenvalue of $A$ and we always let the eigenvalues of a symmetric matrix $A$ be arranged in algebraically nonincreasing order:
$$ \lambda_{\max}(A)=\lambda_{1}(A) \geq \lambda_{2}(A)  \geq \cdots  \geq \lambda_{m}(A).$$ If $A$ and $B$ are symmetric matrices, the relation $A\geq B$ means that $A-B$ is a positive-semidefinite matrix. We use $|\tau|$, $ \mathbb{E}^{k}$, and $ \mathbb{E}$ to denote the number of elements of a set $\mathcal{\tau}$, the conditional expectation conditioned on the first $k$ iterations, and the full expected value, respectively, and let $[m]:=\{1, \cdots, m\}$ for an integer $m\geq 1$.  In addition, we also define $$ h_2(A)=\mathop{\text{inf}}\limits_{x\neq 0}\frac{\|Ax\|_2}{\|x\|_2},\quad \|A\|_{2,\infty}=\mathop{\text{max}}\limits_{1\leq i\leq m}\|A^{(i)}\|_2,$$
 $$u(x)=\left[ \frac{  f_1(x)}{\|\nabla f_1(x)\|_2}, \cdots,  \frac{  f_m(x)}{\|\nabla f_m(x)\|_2}\right]^{T}\in \mathbb{R}^{m },$$
 $$\text{and}\quad G(x)=\left[ \frac{\nabla f_1(x)}{\|\nabla f_1(x)\|_2}, \cdots,  \frac{\nabla f_m(x)}{\|\nabla f_m(x)\|_2}\right]^{T}\in \mathbb{R}^{m\times n}  .$$

\subsection{Previous results}

We first list the nonlinear Kaczmarz method proposed in \cite{wang2022nonlinear} in \Cref{alg:splitting}. Three different rules selecting $i_k$ were considered in \cite{wang2022nonlinear}, which lead to the following three different methods.
\begin{algorithm}
\caption{Nonlinear Kaczmarz method for nonlinear equations (\ref{Sec11})}
\label{alg:splitting}
\begin{algorithmic}[1]
\STATE{Input: The initial estimate $x_0\in \mathbb{R}^n$. }
\FOR{$k=0, 1, 2, \cdots $ until convergence,}
\STATE{Select an index $i_k\in [ m ]$.}
\STATE{Update $x_{k+1}=x_{k}-\frac{f_{i_k}(x_k)}{\|\nabla f_{i_k}(x_k)\|^2_2}\nabla f_{i_k}(x_k)$. }
\ENDFOR
\end{algorithmic}
\end{algorithm}
\begin{itemize}
\item [$(1)$] Nonlinear randomized Kaczmarz (NRK) method: $i_k$ is randomly chosen from $[ m ]$ with probability of
\begin{align}
p_i= \frac{|f_i(x_k)|^2}{\|f(x_k)\|^2_2}.    \notag
\end{align}
\item [$(2)$] Nonlinear Kaczmarz (NK) method: $i_k$ cyclically picks value from $[ m ]$.
\item [$(3)$] Nonlinear uniformly randomized Kaczmarz (NURK) method: $i_k$ is randomly sampled from $[ m ]$ with equal probability.
\end{itemize}

The convergence rate of the NK method is roughly equivalent to that of the NURK method when $m$ is very large 
\cite{wang2022nonlinear},
and the NRK and NURK methods also have the same convergence rate listed in \Cref{lemma01}.
\begin{lemma}[\cite{wang2022nonlinear}]\label{lemma01}
If $f^{\prime}(x)$ is a full column rank matrix and nonlinear function $f $ satisfies the local tangential cone condition given in \Cref{definition}, then the NRK and NURK methods satisfy
\begin{align}
\mathbb{E}\left[\left\| {x}_{k+1}- {x}_{\star}\right\|^{2}_2\right] \leq\left(1-\frac{1-2 \eta}{(1+\eta)^{2}   }
 \frac{h_2^2(f^{\prime}(x_{k}))}{  \|f^{\prime}(x_{k}) \|_{F}^{2}m} \right)
 \mathbb{E}\left[\left\|x_{k}-{x}_{\star}\right\|^{2}_2\right].\notag
 \end{align}
\end{lemma}

In addition, the following definition and lemmas are also necessary throughout the paper.

\begin{definition}[\cite{2017Kaczmarz}]\label{definition}
If for $i \in [ m ]$ and $\forall  x_{1}, x_{2} \in \mathbb{R}^{n}$, there exists $\eta_{i} \in[0, \eta)$ satisfying $\eta=\max\limits _{i}
\eta_{i}<\frac{1}{2}$ such that
\begin{align}
\left|f_{i}\left({x}_{1}\right)-f_{i}\left({x}_{2}\right)-\nabla f_{i}\left({x}_{1}\right)^T\left({x}_{1}-{x}_{2}\right)\right|
\leq \eta_{i}\left|f_{i}\left({x}_{1}\right)-f_{i}\left({x}_{2}\right)\right|,\label{Sec2.3}
\end{align}
then the function $f: \mathbb{R}^{n} \rightarrow \mathbb{R}^{m}$ is referred to satisfy the local tangential cone condition.
\end{definition}

\begin{lemma}[\cite{wang2022nonlinear}]\label{lemma1}
If the function $f$ satisfies the local tangential cone condition, then for $i \in [ m ]$, $\forall x_{1}, x_{2} \in \mathbb{R}^{n}$ and the updating formula (\ref{Sec1.4}), we have
\begin{align}
\left\|x_{k+1}-x_{\star}\right\|^{2}_2 \leq \left\|x_{k}-x_{\star}\right\|^{2}_2-\left(1-2 \eta_{i}\right)
\frac{\left|f_{i}(x_k)\right|^{2}}{\left\|\nabla f_{i}\left(x_{k}\right)\right\|_2^{2}}.\notag 
\end{align}
\end{lemma}

\begin{lemma}[\cite{golub2013matrix}] \label{lemma2}
If $A$ and $A+E$ are $n$-by-$n$ symmetric matrices, then
$$
\lambda_{k}(A)+\lambda_{n}(E) \leq \lambda_{k}(A+E) \leq \lambda_{k}(A)+\lambda_{1}(E), \quad k=1: n .
$$
\end{lemma}

 \section{Single sample-based greedy randomized sampling methods}
\label{sec: Single-methods}
In this section, we mainly propose two single sample-based greedy randomized sampling methods. Before that, we first discuss two greedy rules in detail, which are the key of this paper.
\subsection{Two greedy rules}
 \label{subsec: rules}
In \cite{zeng2020successive}, Zeng et al. presented a greedy selection strategy as follows:
\begin{align}
i_k={\rm arg} \max \limits _{1\leq i\leq m} \left|f_{i}(x_k)\right|^2, \label{Sec3.01}
\end{align}
which is aimed at choosing the maximum magnitude entry of the absolute residual vector. They also showed that this greedy strategy converges faster than the NK and NURK methods in both theoretical analysis and experimental results. When $f(x)=Ax-b$, (\ref{Sec3.01}) reduces to:
\begin{align}
i_k={\rm arg} \max \limits _{1\leq i\leq m} \left|A^{(i)}x_k-b^{(i)}\right|^2, \label{Sec3.02}
\end{align}
which is called the maximum residual rule \cite{griebel2012greedy,nutini2016convergence} or Motzkin method \cite{Agamon54,Motzkin54} and widely used in various areas including linear system and linear feasibility problems.

According to 
\Cref{lemma1}, we know that the magnitude of $\left\|x_{k+1}-x_{\star}\right\|^{2} $ is almost entirely determined by $\frac{\left|f_{i}(x_k)\right|^{2}}{\left\|\nabla f_{i}\left(x_{k}\right)\right\|_2^{2}}$ because $\left\|x_{k}-x_{\star}\right\|^{2}$ is a given value. Consequently, 
\Cref{lemma1} implies that to make the most progress in one iteration, we should pick $i_k$ corresponding to the largest $\frac{\left|f_{i}(x_k)\right|^{2}}{\left\|\nabla f_{i}\left(x_{k}\right)\right\|_2^{2}}$ for all $i\in[m]$. Therefore, we propose a new  greedy selection rule as follows:
\begin{align}
i_k={\rm arg} \max \limits _{1 \leq i\leq m} \frac{\left|f_{i}(x_k)\right|^{2}}{\left\|\nabla f_{i}\left(x_{k}\right)\right\|_2^{2}}.\label{Sec3.1}
\end{align}
When $f(x)=Ax-b$, (\ref{Sec3.1}) reduces to:
\begin{align}
i_k={\rm arg} \max \limits _{1\leq i\leq m} \frac{\left|A^{(i)}x_k-b^{(i)}\right|^2}{{\|A^{(i)}\|^2_2}}, \label{Sec3.2}
\end{align}
which is called the maximum distance rule \cite{Gao2019,nutini2016convergence} and also widely used in the Kaczmarz setting for solving linear problems.

From the above analysis, we find that (\ref{Sec3.01}) and (\ref{Sec3.1}) can be viewed as the generalization of (\ref{Sec3.02}) and (\ref{Sec3.2}) from linear problems to nonlinear problems, respectively. So, we still refer to the greedy selection strategies (\ref{Sec3.01}) and (\ref{Sec3.1}) as maximum residual rule and maximum distance rule, respectively.

\subsection{Two greedy sampling methods}
 \label{subsec: methods}

The maximum residual rule (\ref{Sec3.01}) shows that sampling larger entries of the residual should converge faster. However, calculating the whole residual is expensive when the number of samples is very large. From this point we give the maximum residual-based sampling nonlinear Kaczmarz (MR-SNK) method presented as the case 1 in \Cref{alg2}. Specifically, the MR-SNK method operates by uniformly sampling an index subset from among $[m]$, computing the absolute subresidual vector of this subset, and projecting onto the linearization of a single nonlinear equation corresponding to the largest magnitude entry of this subresidual.

On the other hand, as the maximum distance selection rule leads to the best decrease in mean squared error in each iteration, similar to the MR-SNK method, the maximum distance-based sampling nonlinear Kaczmarz (MD-SNK) method is provided as the case 2 in \Cref{alg2}.
\begin{algorithm}
\caption{The (MR/MD)-SNK method }
\label{alg2}
\begin{algorithmic}[1]
\STATE{Input: The initial estimate $x_0\in \mathbb{R}^n$, parameter $\beta \in [m]$. }
\FOR{$k=0, 1, 2, \cdots $ until convergence,}
\STATE{Choose an index subset $\tau_k$ of size $\beta$ uniformly at random from among $[ m ]$.}
\STATE{Switch the greedy sampling strategy to determine two different iterative methods, i.e., the MR-SNK and MD-SNK methods.}
\STATE{\quad $\triangleright$ case 1:  maximum residual (MR) rule
\\
\quad \quad $i_k={\rm arg} \max \limits _{i\in \tau_k} \left|f_{i}(x_k)\right|^2 . $}
\STATE{\quad $\triangleright$ case 2: maximum distance (MD) rule
\\
 \quad\quad $i_k={\rm arg} \max \limits _{i\in \tau_k} \frac{\left|f_{i}(x_k)\right|^{2}}{\left\|\nabla f_{i}\left(x_{k}\right)\right\|_2^{2}} .$ }
\STATE{Update $x_{k+1}=x_{k}-\frac{f_{i_k}(x_k)}{\|\nabla f_{i_k}(x_k)\|^2_2}\nabla f_{i_k}(x_k)$. }
\ENDFOR
\end{algorithmic}
\end{algorithm}

\begin{remark} \label{remark_MR-SNK}
The MR-SNK method has been studied for several specific choices of $\beta$ and the function $f$. If $\beta=1$, the MR-SNK method reduces the NURK method \cite{wang2022nonlinear}. If $\beta=m$, we call the MR-SNK method the maximum residual-based nonlinear Kaczmarz (MR-NK) method, which can be seen as a special variant of the GP method presented in \cite{zeng2020successive}. If $f=Ax-b$, the MR-SNK method gives the SKM method \cite{de2017sampling}, which can in turn recover a variant of randomized Kaczmarz method \cite{Strohmer2009} and the Motzkin method \cite{Motzkin54,nutini2016convergence} by choosing different $\beta$.

 The MD-SNK method also has been studied for several specific choices of $\beta$ and the function $f$. If $\beta=1$, the MD-SNK method gives the NURK method \cite{wang2022nonlinear}. The MD-SNK method with $\beta=m$ and $f=Ax-b$ recovers the maximal weighted residual Kaczmarz method \cite{nutini2016convergence,Gao2019}. In addition, similar to the MR-NK method discussed above, we can obtain a new greedy projection selection method on the basis of the maximum distance rule by setting $\beta=m$ in the MD-SNK method. The new greedy method is called as the maximum distance-based nonlinear Kaczmarz (MD-NK) method in this paper. Similar to the SKM method, we also obtain a new sampling type Kaczmarz method for solving linear problems by setting $f=Ax-b$ in the MD-SNK method, and we call it the sampling distance-based Kaczmarz (SDK) method. Therefore, different parameters $\beta$ and function $f$ in the MR/MD-SNK method lead to different methods, which are summarized in \Cref{Sec3.2:tab1} for convenience.
 \begin{table}[]
\centering
   \fontsize{8}{8}\selectfont
       \caption{Summary of special methods of the MR/MD-SNK method for different parameters $\beta$ and function $f$ in \Cref{alg2}.}
    \label{Sec3.2:tab1}
    \begin{tabular}{c c c c   }
 \hline
           &Method&  $\beta$  &  $f$ \cr \hline
           &NURK \cite{wang2022nonlinear}      & 1                   & $f$       \cr
MR-SNK     &MR-NK                              & m                   & $f$       \cr
           &SKM  \cite{de2017sampling}         & $\beta\in[m]$       & $Ax-b$  \cr\hline
           &NURK \cite{wang2022nonlinear}      & 1                   & $f$       \cr
MD-SNK     &MD-NK                              & m                   & $f$       \cr
           &SDK                                & $\beta\in[m]$       & $Ax-b$  \cr\hline
\end{tabular}
\end{table}
\end{remark}

\section{Convergence analysis}
\label{sec:theory}
In this section, we show that the MR-SNK and MD-SNK methods enjoy a linear rate in expectation of convergence. Before that, we first present \Cref{lemma3,lemma4}, which are crucial to the convergence analysis in the following theorems.

\begin{lemma}  \label{lemma3}
If the function $f $ satisfies the local tangential cone condition given in \Cref{definition}, for $\forall x_{1}, x_{2} \in \mathbb{R}^{n}$ and an index subset $\tau\subseteq[m]$, we have
\begin{align}
\left\|f_{\tau}\left({x}_{1}\right)-f_{\tau}\left({x}_{2}\right)-f^{\prime}_{\tau}\left({x}_{1}\right)\left({x}_{1}-{x}_{2}\right)\right\|^2_2
\leq \eta^2\left\|f_{\tau}\left({x}_{1}\right)-f_{\tau}\left({x}_{2}\right)\right\|^2_2\label{Sec2.6}
\end{align}
and
\begin{align}
\left\|f_{\tau}\left({x}_{1}\right)-f_{\tau}\left({x}_{2}\right)\right\|^2_2
\geq \frac{1}{1+\eta^2}\left\|f^{\prime}_{\tau}\left(x_{1}\right)\left(x_{1}-x_{2}\right)\right\|^2_2.\label{Sec2.7}
\end{align}
\end{lemma}

\begin{proof}
From (\ref{Sec2.3}) in \Cref{definition}, we have
\begin{align}
\left\|f_{\tau}\left({x}_{1}\right)-f_{\tau}\left({x}_{2}\right)-f^{\prime}_{\tau}\left({x}_{1}\right)\left({x}_{1}-{x}_{2}\right)\right\|^2_2
&=
\sum\limits_{i\in\tau}\left|f_{i}\left({x}_{1}\right)-f_{i}\left({x}_{2}\right)-\nabla f_{i}\left({x}_{1}\right)^T\left({x}_{1}-{x}_{2}\right)\right|^2 \notag
\\
&\leq \sum\limits_{i\in\tau} \eta_{i}^2\left|f_{i}\left({x}_{1}\right)-f_{i}\left({x}_{2}\right)\right|^2\notag
\\
&\leq  \eta^2\left\|f_{\tau}\left({x}_{1}\right)-f_{\tau}\left({x}_{2}\right)\right\|^2_2,\notag
\end{align}
which is the desired result (\ref{Sec2.6}).

According to (\ref{Sec2.6}), we get
\begin{align}
\eta^2\left\|f_{\tau}\left({x}_{1}\right)-f_{\tau}\left({x}_{2}\right)\right\|^2_2
&\geq\left\|f_{\tau}\left({x}_{1}\right)-f_{\tau}\left({x}_{2}\right)-f^{\prime}_{\tau}\left({x}_{1}\right)\left({x}_{1}-{x}_{2}\right)\right\|^2_2\notag
\\
&\geq \left\|f^{\prime}_{\tau}\left({x}_{1}\right)\left({x}_{1}-{x}_{2}\right)\right\|^2_2-\left\|f_{\tau}\left({x}_{1}\right)-f_{\tau}\left({x}_{2}\right)\right\|^2_2,\notag
\end{align}
which implies
\begin{align}
(1+\eta^2)\left\|f_{\tau}\left({x}_{1}\right)-f_{\tau}\left({x}_{2}\right)\right\|^2_2
&\geq \left\|f^{\prime}_{\tau}\left({x}_{1}\right)\left({x}_{1}-{x}_{2}\right)\right\|^2_2,\notag
\end{align}
that is,
\begin{align}
\left\|f_{\tau}\left({x}_{1}\right)-f_{\tau}\left({x}_{2}\right)\right\|^2_2
\geq \frac{1}{1+\eta^2}\left\|f^{\prime}_{\tau}\left(x_{1}\right)\left(x_{1}-x_{2}\right)\right\|^2_2.\notag
\end{align}
\end{proof}

\begin{lemma}  \label{lemma4}
If the function $f $ satisfies the local tangential cone condition given in \Cref{definition} and a vector $x_\star\in \mathbb{R}^{n} $ satisfies $f(x_\star)=0$, then for $\forall x\in \mathbb{R}^{n}$ and an index subset $\tau\subseteq[m]$, we have
\begin{align}
\left\|u_{\tau}\left({x}\right)-G_{\tau}\left({x}\right)\left({x}-{x}_{\star}\right)\right\|^2_2
\leq \eta^2\left\|u_{\tau}\left({x}\right)\right\|^2_2\label{Sec2.6.1}
\end{align}
and
\begin{align}
\left\|u_{\tau}\left({x}\right)\right\|^2_2
\geq \frac{1}{1+\eta^2}\left\|G_{\tau}\left(x\right)\left(x-x_{\star}\right)\right\|^2_2.\label{Sec2.7.1}
\end{align}
\end{lemma}

\begin{proof}
Considering the definitions of $u(x)$ and $G(x)$ given in \Cref{subsec:notation} and (\ref{Sec2.3}) in \Cref{definition}, we have
\begin{align}
\left\|u_{\tau}\left({x}\right)-G_{\tau}\left({x}\right)\left({x}-{x}_{\star}\right)\right\|^2_2
&=
\sum\limits_{i\in\tau}\left|u^{(i)}\left({x}\right)-G^{(i)}\left({x}\right)\left({x}-{x}_{\star}\right)\right|^2_2\notag
\\
&=
\sum\limits_{i\in\tau} \frac{\left | f_i(x)-\nabla f_i(x)^T\left({x}-{x}_{\star}\right)\right|^2}{\|\nabla f_i(x)\|^2_2} \notag
\\
&=
\sum\limits_{i\in\tau} \frac{\left | f_i(x)-f_i(x_{\star})-\nabla f_i(x)^T\left({x}-{x}_{\star}\right)\right|^2}{\|\nabla f_i(x)\|^2_2} \notag
\\
&\leq
\sum\limits_{i\in\tau} \frac{\eta_{i}^2\left | f_{i}\left({x}\right)-f_{i}\left({x_{\star}}\right)\right|^2}{\|\nabla f_i(x)\|^2_2} \notag
\\
&\leq
\sum\limits_{i\in\tau} \eta^2\frac{\left | f_{i}\left({x}\right)\right|^2}{\|\nabla f_i(x)\|^2_2} \notag
\\
&=
\eta^2\left\|u_{\tau}\left({x}\right)\right\|^2_2, \notag
\end{align}
which is the result (\ref{Sec2.6.1}).

On the other hand, by applying (\ref{Sec2.6.1}) we get
\begin{align}
\eta^2\left\|u_{\tau}\left({x}\right)\right\|^2_2
&\geq
\left\|u_{\tau}\left({x}\right)-G_{\tau}\left({x}\right)\left({x}-{x}_{\star}\right)\right\|^2_2\notag
\\
&\geq
\left\|G_{\tau}\left({x}\right)\left({x}-{x}_{\star}\right)\right\|^2_2-\left\|u_{\tau}\left({x}\right)\right\|^2_2,\notag
\end{align}
which implies
\begin{align}
\left(1+\eta^2\right)\left\|u_{\tau}\left({x}\right)\right\|^2_2
\geq
\left\|G_{\tau}\left({x}\right)\left({x}-{x}_{\star}\right)\right\|^2_2,\notag
\end{align}
that is,
\begin{align}
\left\|u_{\tau}\left({x}\right)\right\|^2_2
\geq \frac{1}{1+\eta^2}\left\|G_{\tau}\left(x\right)\left(x-x_{\star}\right)\right\|^2_2.\notag
\end{align}

\end{proof}

\begin{theorem}
\label{theorem:MR-SNK}
If the nonlinear function $f $ satisfies the local tangential cone condition given in \Cref{definition}, $\eta=\max\limits _{i}
\eta_{i}<\frac{1}{2}$, and $f(x_{\star})=0$, then the iterations of MR-SNK method, i.e., the first case in \Cref{alg2}, satisfy
\begin{align}
\mathbb{E}\left[\left\| {x}_{k+1}- {x}_{\star}\right\|^{2}_2\right] \leq\left(1 - \frac{ 1-2 \eta }{1+\eta^2} \frac{\beta}{\xi_k}  \frac{ h_2^2\left( f^{\prime}\left(x_{k}\right)\right )}{\left\|f^{\prime}\left(x_{k}\right)\right\|_{2,\infty}^{2} m}  \right)
 \mathbb{E}\left[\left\|x_{k}-{x}_{\star}\right\|^{2}_2\right],\label{Sec4.theorem1}
 \end{align}
 where
\begin{align}
\xi_k=\frac{\sum  \limits _{\tau_k \in \binom{[m]}{\beta}} \left \|f_{\tau_k}(x_k)\right\|_{2}^{2}}{\sum  \limits _{\tau_k \in \binom{[m]}{\beta}}  \left\|f_{\tau_k}(x_k)\right\|_{\infty}^{2}}.\label{Sec4.theorem1.1}
\end{align}
\end{theorem}

\begin{proof}
According to the MR-SNK method in \Cref{alg2} and 
\Cref{lemma1}, we know that
\begin{align}
\left\|x_{k+1}-x_{\star}\right\|^{2}_2
&\leq \left\|x_{k}-x_{\star}\right\|^{2}_2-\left(1-2 \eta_{i_k}\right)
\frac{\left|f_{i_k}(x_k)\right|^{2}}{\left\|\nabla f_{i_k}\left(x_{k}\right)\right\|_2^{2}} \notag
\\
&= \left\|x_{k}-x_{\star}\right\|^{2}_2 -\left(1-2 \eta_{i_k}\right)
\frac{\left\|f_{\tau_k}(x_k)\right\|^{2}_{\infty} }{\left\|\nabla f_{i_k}\left(x_{k}\right)\right\|_2^{2}}. \notag
\end{align}
Taking expectation of both sides conditioned on $x_{k}$ gives
\begin{align}
\mathbb{E}^{k}\left[\left\| {x}_{k+1}- {x}_{\star}\right\|^{2}_2\right]
&\leq \left\|x_{k}-x_{\star}\right\|^{2}_2 -\mathbb{E}^{k}\left[\left(1-2 \eta_{i_k}\right)
\frac{\left\|f_{\tau_k}(x_k)\right\|^{2}_{\infty} }{\left\|\nabla f_{i_k}\left(x_{k}\right)\right\|_2^{2}}\right] \notag
\\
&=\left\|x_{k}-x_{\star}\right\|^{2}_2 - \sum\limits_{\tau_k\in\binom{[m]}{\beta}}\frac{1}{\binom{m}{\beta}}\left(1-2 \eta_{i_k}\right)
\frac{\left\|f_{\tau_k}(x_k)\right\|^{2}_{\infty} }{\left\|\nabla f_{i_k}\left(x_{k}\right)\right\|_2^{2}}\notag
\\
&\leq \left\|x_{k}-x_{\star}\right\|^{2}_2 -\frac{1}{\binom{m}{\beta}}\left(1-2 \eta\right)\sum\limits_{\tau_k\in\binom{[m]}{\beta}}
\frac{\left\|f_{\tau_k}(x_k)\right\|^{2}_{\infty} }{\left\|\nabla f_{i_k}\left(x_{k}\right)\right\|_2^{2}} \notag
\\
&\leq\left\|x_{k}-x_{\star}\right\|^{2}_2 -\frac{1}{\binom{m}{\beta}}\frac{ 1-2 \eta }{\left\|f^{\prime}\left(x_{k}\right)\right\|_{2,\infty}^{2}} \sum\limits_{\tau_k\in\binom{[m]}{\beta}}
\left\|f_{\tau_k}(x_k)\right\|^{2}_{\infty},   \notag
\end{align}
which together with $\xi_k$ defined in (\ref{Sec4.theorem1.1}) gets
\begin{align}
\mathbb{E}^{k}\left[\left\| {x}_{k+1}- {x}_{\star}\right\|^{2}_2\right]
&\leq\left\|x_{k}-x_{\star}\right\|^{2}_2 -\frac{1}{\binom{m}{\beta}}\frac{ 1-2 \eta }{\left\|f^{\prime}\left(x_{k}\right)\right\|_{2,\infty}^{2}} \frac{1}{\xi_k} \sum\limits_{\tau_k\in\binom{[m]}{\beta}}
\left\|f_{\tau_k}(x_k)\right\|^{2}_{2}   \notag
\\
&=\left\|x_{k}-x_{\star}\right\|^{2}_2 -\frac{1}{\binom{m}{\beta}}\frac{ 1-2 \eta }{\left\|f^{\prime}\left(x_{k}\right)\right\|_{2,\infty}^{2}} \frac{1}{\xi_k}  \frac{\binom{m}{\beta}\beta}{m}\left\|f(x_k)\right\|^{2}_{2} \notag
\\
&=\left\|x_{k}-x_{\star}\right\|^{2}_2 - \frac{ 1-2 \eta }{\left\|f^{\prime}\left(x_{k}\right)\right\|_{2,\infty}^{2}} \frac{1}{\xi_k}  \frac{ \beta}{m}\left\|f(x_k)-f(x_{\star})\right\|^{2}_{2}. \notag
\end{align}
Further, considering (\ref{Sec2.7}) in \Cref{lemma3} and the definition of $h_2(\cdot)$ in \Cref{subsec:notation}, we get
\begin{align}
\mathbb{E}^{k}\left[\left\| {x}_{k+1}- {x}_{\star}\right\|^{2}_2\right]
&\leq\left\|x_{k}-x_{\star}\right\|^{2}_2 - \frac{ 1-2 \eta }{\left\|f^{\prime}\left(x_{k}\right)\right\|_{2,\infty}^{2}} \frac{1}{\xi_k}  \frac{ \beta}{m}\frac{1}{1+\eta^2}\left\|f^{\prime}\left(x_{k}\right)\left(x_{k}-x_{\star}\right)\right\|^{2}_{2}\notag
\\
&\leq\left\|x_{k}-x_{\star}\right\|^{2}_2 - \frac{ 1-2 \eta }{\left\|f^{\prime}\left(x_{k}\right)\right\|_{2,\infty}^{2}} \frac{1}{\xi_k}  \frac{ \beta}{m}\frac{1}{1+\eta^2} h_2^2\left( f^{\prime}\left(x_{k}\right)\right )\left\| x_{k}-x_{\star}\right\|^{2}_{2} \notag
\\
&=\left(1 - \frac{ 1-2 \eta }{1+\eta^2} \frac{\beta}{\xi_k}  \frac{ h_2^2\left( f^{\prime}\left(x_{k}\right)\right )}{\left\|f^{\prime}\left(x_{k}\right)\right\|_{2,\infty}^{2} m}  \right)\left\|x_{k}-x_{\star}\right\|^{2}_{2}. \notag
\end{align}
So, taking expectation on both sides and using the tower rule of expectation, we get the desired result (\ref{Sec4.theorem1}).

\end{proof}

\begin{remark}\label{remark_theorem_MR-SNK-1}
Noting that $1\leq\xi_k\leq\beta$ and $\beta h_2^2\left( f^{\prime}\left(x_{k}\right)\right ) \leq\left\|f^{\prime}\left(x_{k}\right)\right\|_{2,\infty}^{2} m$, it is easy to obtain that
\begin{align}
0<\frac{ 1-2 \eta }{1+\eta^2} \frac{1}{\xi_k}  \frac{\beta h_2^2\left( f^{\prime}\left(x_{k}\right)\right )}{\left\|f^{\prime}\left(x_{k}\right)\right\|_{2,\infty}^{2} m}  <1, \notag
\end{align}
which indicates that the convergence factor of the MR-SNK method, i.e.,
\begin{align}
\rho_{\text{MR-SNK}}=1-\frac{ 1-2 \eta }{1+\eta^2} \frac{\beta}{\xi_k}  \frac{ h_2^2\left( f^{\prime}\left(x_{k}\right)\right )}{\left\|f^{\prime}\left(x_{k}\right)\right\|_{2,\infty}^{2} m}, \notag
\end{align}
is smaller than 1. Thus, the iteration of the MR-SNK method is contracted and convergent in expectation.
\end{remark}

\begin{remark}\label{remark_theorem_MR-SNK-2}
According to \Cref{lemma01}, we know that
the convergence factors of the NRK and NURK methods are the same and
\begin{align}
\rho_{\text{NRK}}=\rho_{\text{NURK}}= 1-\frac{1-2 \eta}{(1+\eta)^{2}   }
 \frac{h_2^2(f^{\prime}(x_{k}))}{ \|f^{\prime}(x_{k}) \|_{F}^{2}m   }.\notag
 \end{align}
Since $\|f^{\prime}(x_{k}) \|_{F}^{2}\geq \left\|f^{\prime}\left(x_{k}\right)\right\|_{2,\infty}^{2}$, $\frac{\beta}{\xi_k}\geq 1$ and $0<\eta<\frac{1}{2}$, we have
 \begin{align}
 \frac{ 1-2 \eta }{1+\eta^2} \frac{\beta}{\xi_k}  \frac{ h_2^2\left( f^{\prime}\left(x_{k}\right)\right )}{\left\|f^{\prime}\left(x_{k}\right)\right\|_{2,\infty}^{2} m}
 >
  \frac{1-2 \eta}{(1+\eta)^{2}   }
 \frac{h_2^2(f^{\prime}(x_{k}))}{\|f^{\prime}(x_{k}) \|_{F}^{2} m },\notag
 \end{align}
 which implies that the convergence factor of the MR-SNK method is smaller than that of the NRK and NURK methods. That is,
 \begin{align}
\rho_{\text{MR-SNK}}<\rho_{\text{NRK}}=\rho_{\text{NURK}}.\notag
 \end{align}
\end{remark}

\begin{theorem}
\label{theorem:MD-SNK}
If the nonlinear function $f $ satisfies the local tangential cone condition given in \Cref{definition}, $\eta=\max\limits _{i}
\eta_{i}<\frac{1}{2}$, and $f(x_{\star})=0$, then the iterations of MD-SNK method, i.e., the second case in \Cref{alg2}, satisfy
\begin{align}
\mathbb{E}\left[\left\| {x}_{k+1}- {x}_{\star}\right\|^{2}_2\right]
 \leq
\left(1-  \frac{1-2\eta}{1+\eta^2}\frac{\beta}{\zeta_k} \frac{ h_2^2(G \left(x_k\right))}{m}\right)
 \mathbb{E}\left[\left\|x_{k}-{x}_{\star}\right\|^{2}_2\right],\label{Sec4.theorem2}
 \end{align}
  where
\begin{align}
\zeta_k=\frac{\sum  \limits _{\tau_k \in \binom{[m]}{\beta}} \left \|u_{\tau_k}(x_k)\right\|_{2}^{2}}{\sum  \limits _{\tau_k \in \binom{[m]}{\beta}}  \left\|u_{\tau_k}(x_k)\right\|_{\infty}^{2}}.\label{Sec4.theorem2.1}
\end{align}
\end{theorem}

\begin{proof}
According to 
\Cref{lemma1}, the definition of $u(x)$ in \Cref{subsec:notation} and the MD-SNK method in \Cref{alg2}, we have
\begin{align}
\left\|x_{k+1}-x_{\star}\right\|^{2}_2
&\leq \left\|x_{k}-x_{\star}\right\|^{2}_2-\left(1-2 \eta_{i_k}\right)
\frac{\left|f_{i_k}(x_k)\right|^{2}}{\left\|\nabla f_{i_k}\left(x_{k}\right)\right\|_2^{2}} \notag
\\
&= \left\|x_{k}-x_{\star}\right\|^{2}_2 -\left(1-2 \eta_{i_k}\right) \left|u_{i_k}(x_k)\right|^{2}  \notag
\\
&= \left\|x_{k}-x_{\star}\right\|^{2}_2 -\left(1-2 \eta_{i_k}\right) \left\|u_{\tau_k}(x_k)\right\|^{2}_\infty . \notag
\end{align}
Now, taking expectation of both sides conditioned on $x_{k}$ gives
\begin{align}
\mathbb{E}^{k}\left[\left\| {x}_{k+1}- {x}_{\star}\right\|^{2}_2\right]
&\leq \left\|x_{k}-x_{\star}\right\|^{2}_2 -\mathbb{E}^{k}\left[\left(1-2 \eta_{i_k}\right)
 \left\|u_{\tau_k}(x_k)\right\|^{2}_\infty\right] \notag
\\
&=\left\|x_{k}-x_{\star}\right\|^{2}_2 - \sum\limits_{\tau_k\in\binom{[m]}{\beta}}\frac{1}{\binom{m}{\beta}}\left(1-2 \eta_{i_k}\right)
\left\|u_{\tau_k}(x_k)\right\|^{2}_\infty\notag
\\
&\leq \left\|x_{k}-x_{\star}\right\|^{2}_2 -\frac{1}{\binom{m}{\beta}}\left(1-2 \eta\right)\sum\limits_{\tau_k\in\binom{[m]}{\beta}}
\left\|u_{\tau_k}(x_k)\right\|^{2}_\infty ,   \notag
\end{align}
which together with the definition of $\zeta_k$ in (\ref{Sec4.theorem2.1}) leads to
\begin{align}
\mathbb{E}^{k}\left[\left\| {x}_{k+1}- {x}_{\star}\right\|^{2}_2\right]
&\leq
\left\|x_{k}-x_{\star}\right\|^{2}_2 -\frac{1}{\binom{m}{\beta}}\left(1-2 \eta\right)\frac{1}{\zeta_k} \sum\limits_{\tau_k\in\binom{[m]}{\beta}}
\left\|u_{\tau_k}(x_k)\right\|^{2}_2   \notag
\\
&=
\left\|x_{k}-x_{\star}\right\|^{2}_2 -\frac{1}{\binom{m}{\beta}}\left(1-2 \eta\right)\frac{1}{\zeta_k} \frac{\binom{m}{\beta}\beta}{m}
\left\|u(x_k)\right\|^{2}_2   \notag
\\
&=
\left\|x_{k}-x_{\star}\right\|^{2}_2 - \left(1-2 \eta\right)\frac{\beta}{\zeta_k} \frac{ 1}{m}
\left\|u(x_k)\right\|^{2}_2  . \notag
\end{align}
Further, by making use of (\ref{Sec2.7.1}) in \Cref{lemma4} and the definition of $h_2(\cdot)$ in \Cref{subsec:notation}, we can obtain
\begin{align}
\mathbb{E}^{k}\left[\left\| {x}_{k+1}- {x}_{\star}\right\|^{2}_2\right]
&\leq
\left\|x_{k}-x_{\star}\right\|^{2}_2 - \left(1-2 \eta\right)\frac{\beta}{\zeta_k} \frac{ 1}{m}
\frac{1}{1+\eta^2}\left\|G \left(x_k\right)\left(x_k-x_{\star}\right)\right\|^2_2  \notag
\\
&\leq
\left(1-  \frac{1-2\eta}{1+\eta^2}\frac{\beta}{\zeta_k} \frac{ h_2^2(G \left(x_k\right))}{m}\right)
\left\|\left(x_k-x_{\star}\right)\right\|^2_2 . \notag
\end{align}
Finally, the desired result (\ref{Sec4.theorem2}) can be obtained by taking the full expectation on both sides.

\end{proof}

\begin{remark}\label{remark_theorem_MD-SNK-1}
Considering that $1\leq\zeta_k\leq\beta$ and $\beta h_2^2\left( G\left(x_{k}\right)\right ) \leq m$, we get
\begin{align}
0<\frac{ 1-2 \eta }{1+\eta^2} \frac{1}{\zeta_k}  \frac{\beta h_2^2\left(G\left(x_{k}\right)\right )}{  m}  <1, \notag
\end{align}
which indicates that the convergence factor of the MD-SNK method, i.e.,
\begin{align}
\rho_{\text{MD-SNK}}= 1-  \frac{1-2\eta}{1+\eta^2}\frac{\beta}{\zeta_k} \frac{ h_2^2(G \left(x_k\right))}{m}, \notag
\end{align}
is smaller than 1.
\end{remark}

\begin{remark}\label{remark_theorem_MD-SNK-2}
Letting $D$ be an $m$-by-$m$ diagonal matrix in the following form 
 $$
 D=\begin{bmatrix}
 \|\nabla f_1(x_k)\|_2&&\\
 &\ddots&\\
 && \|\nabla f_m(x_k)\|_2
 \end{bmatrix},
 $$
then we get a relation between $f^{\prime}(x_k)$ and $G(x_k)$, that is,
\begin{align}
f^{\prime}(x_k)=DG(x_k).\label{remark_theorem_MD-SNK-2:1}
 \end{align}
Now, by setting $A=G(x_k)G(x_k)^T$ and $B=D^TD$, since both matrices $A$ and $B$ are symmetric positive-semidefinite, we have
\begin{align}
\lambda_1(B)I\geq B\geq \lambda_m(B)I\notag
 \end{align}
and then obtan
\begin{align}
\lambda_1(B)A=A^{\frac{1}{2}}(\lambda_1(B)I-B)A^{\frac{1}{2}}+A^{\frac{1}{2}}BA^{\frac{1}{2}}\geq A^{\frac{1}{2}}BA^{\frac{1}{2}}.\notag
 \end{align}
Further, by applying \Cref{lemma2}, we get
\begin{align}
\lambda_m\left(\lambda_1(B)A\right)
&=\lambda_m\left(A^{\frac{1}{2}}(\lambda_1(B)I-B)A^{\frac{1}{2}}+A^{\frac{1}{2}}BA^{\frac{1}{2}}\right)\notag
\\
&\geq\lambda_m\left(A^{\frac{1}{2}}(\lambda_1(B)I-B)A^{\frac{1}{2}} \right)+\lambda_m\left(A^{\frac{1}{2}}BA^{\frac{1}{2}}\right)\notag
\\
&\geq \lambda_m\left(A^{\frac{1}{2}}BA^{\frac{1}{2}}\right) \notag\\
&=\lambda_m\left(AB\right),\notag
 \end{align}
which implies
\begin{align}
\lambda_1(B)\lambda_m\left(A\right)
\geq  \lambda_m\left(AB\right).\notag
 \end{align}
 That is,
 \begin{align}
\lambda_1(D^TD)\lambda_m\left(G(x_k)G(x_k)^T\right)
\geq \lambda_m\left(G(x_k)G(x_k)^TD^TD\right)=\lambda_m\left(DG(x_k)G(x_k)^TD^T\right),\notag
 \end{align}
 which together with the definitions of $D$, $G(x_k)$ and $f^{\prime}(x_k)=DG(x_k)$ in (\ref{remark_theorem_MD-SNK-2:1}) indicates
  \begin{align}
\left\|f^{\prime}\left(x_{k}\right)\right\|_{2,\infty}^{2} h_2^2(G \left(x_k\right))
\geq h_2^2\left( f^{\prime}\left(x_{k}\right)\right ).\label{remark_theorem_MD-SNK-2:2}
 \end{align}
Meanwhile, from \Cref{theorem:MR-SNK,theorem:MD-SNK}, we know that the convergence factors of the MR-SNK and MD-SNK methods are
$$
\rho_{\text{MR-SNK}}\sim 1 - \frac{ 1-2 \eta }{1+\eta^2}   \frac{ h_2^2\left( f^{\prime}\left(x_{k}\right)\right )}{\left\|f^{\prime}\left(x_{k}\right)\right\|_{2,\infty}^{2} m},
\quad
\text{and}
\quad
\rho_{\text{MD-SNK}}\sim 1-  \frac{1-2\eta}{1+\eta^2}  \frac{ h_2^2(G \left(x_k\right))}{m},
$$
respectively.

Thus, according to (\ref{remark_theorem_MD-SNK-2:2}), we conclude that
\begin{align}
  1-  \frac{1-2\eta}{1+\eta^2}  \frac{ h_2^2(G \left(x_k\right))}{m}
\leq
   1 - \frac{ 1-2 \eta }{1+\eta^2}   \frac{ h_2^2\left( f^{\prime}\left(x_{k}\right)\right )}{\left\|f^{\prime}\left(x_{k}\right)\right\|_{2,\infty}^{2} m}, \label{remark_theorem_MD-SNK-2:3}
\end{align}
which means that the convergence factor of the MD-SNK method is a smaller than that of the MR-SNK method.
\end{remark}

\begin{remark}\label{remark_theorem_MD-SNK-3}
From \Cref{remark_MR-SNK}, we know that the MR-SNK and MD-SNK methods can recover the NURK method by setting $\beta=1$. It is convenient to obtain the convergence result of the NURK method by using \Cref{theorem:MR-SNK} or \Cref{theorem:MD-SNK}. Specifically, substituting $\beta=\xi_k=\zeta_k=1$ into the convergence bounds (\ref{Sec4.theorem1}) and (\ref{Sec4.theorem2}) and combining with (\ref{remark_theorem_MD-SNK-2:3}) yields
\begin{align}
\mathbb{E}\left[\left\| {x}_{k+1}- {x}_{\star}\right\|^{2}_2\right]
&\leq\left(1-  \frac{1-2\eta}{1+\eta^2}  \frac{ h_2^2(G \left(x_k\right))}{m} \right)
 \mathbb{E}\left[\left\|x_{k}-{x}_{\star}\right\|^{2}_2\right] \quad(by~ \Cref{theorem:MD-SNK})   \notag
 \\
 &\leq\left(1 - \frac{ 1-2 \eta }{1+\eta^2} \frac{ h_2^2\left( f^{\prime}\left(x_{k}\right)\right )}{\left\|f^{\prime}\left(x_{k}\right)\right\|_{2,\infty}^{2} m}  \right)
 \mathbb{E}\left[\left\|x_{k}-{x}_{\star}\right\|^{2}_2\right] \quad(by~ \Cref{theorem:MR-SNK})   \notag
 \\
 &<\left(1 - \frac{ 1-2 \eta }{(1+\eta)^2} \frac{ h_2^2\left( f^{\prime}\left(x_{k}\right)\right )}{\left\|f^{\prime}\left(x_{k}\right)\right\|_{F}^{2} m}  \right)
 \mathbb{E}\left[\left\|x_{k}-{x}_{\star}\right\|^{2}_2\right] \quad(by~Theorem~3.1~ in~\text{\cite{wang2022nonlinear}}),   \notag
 \end{align}
 which indicates that our results are all tighter than the existing error estimate for the NURK method.
\end{remark}

\begin{remark}\label{remark_theorem_MD-SNK-4}
\Cref{theorem:MR-SNK,theorem:MD-SNK} show that the progress made by an iteration of the MR-SNK and MD-SNK methods depends on the value of $\xi_k$ and $\zeta_k$, respectively. Since $1\leq\xi_k \leq \beta$, we can conclude that the closer $\xi_k$ approaches 1, the smaller the convergence factor of the MR-SNK method is, and we call it the ``best case'' when $\xi_k =1$. Conversely, when $\xi_k$ is closer to $\beta$,  the convergence factor of MR-SNK method is bigger, and we call it the ``worst case'' when $\xi_k =\beta$. Similarly, we can get other convergence results in the best and worst cases, which are summarised in \Cref{Sec4:tab1}. Moreover, we summarize the relationships of these best case convergence factors in \Cref{Sec4:tab2}. The inequality sign $(\leq)$ denotes that the best case convergence factor of the method listed in the row is smaller or equal to the method listed in the column, and ``P'' indicates that the smaller convergence factor is problem-dependent.
\begin{table}[]
\centering
   \fontsize{8}{8}\selectfont
       \caption{The best and worst case convergence factors of the MR-SNK, MR-NK, NURK-MR (result of the NURK method from \Cref{theorem:MR-SNK}), MD-SNK, MD-NK, NURK-MD (result of the NURK method from \Cref{theorem:MD-SNK}) methods.}
    \label{Sec4:tab1}
    \begin{tabular}{ c c c   }
 \hline
            Method  &     Worst case                 &  Best case                    \cr \hline
            MR-SNK   &   $1 - \frac{ 1-2 \eta }{1+\eta^2}   \frac{ h_2^2\left( f^{\prime}\left(x_{k}\right)\right )}{\left\|f^{\prime}\left(x_{k}\right)\right\|_{2,\infty}^{2} m} $  &   $1 - \frac{ 1-2 \eta }{1+\eta^2}   \frac{\beta h_2^2\left( f^{\prime}\left(x_{k}\right)\right )}{\left\|f^{\prime}\left(x_{k}\right)\right\|_{2,\infty}^{2} m} $                               \cr
            MR-NK   &  $ 1 - \frac{ 1-2 \eta }{1+\eta^2} \frac{ h_2^2\left( f^{\prime}\left(x_{k}\right)\right )}{\left\|f^{\prime}\left(x_{k}\right)\right\|_{2,\infty}^{2} m }                     $       &  $ 1 - \frac{ 1-2 \eta }{1+\eta^2} \frac{ h_2^2\left( f^{\prime}\left(x_{k}\right)\right )}{\left\|f^{\prime}\left(x_{k}\right)\right\|_{2,\infty}^{2}  }  $   \cr
            NURK-MR &   $ 1 - \frac{ 1-2 \eta }{1+\eta^2}  \frac{ h_2^2\left( f^{\prime}\left(x_{k}\right)\right )}{\left\|f^{\prime}\left(x_{k}\right)\right\|_{2,\infty}^{2} m}          $                 &  $1 - \frac{ 1-2 \eta }{1+\eta^2} \frac{ h_2^2\left( f^{\prime}\left(x_{k}\right)\right )}{\left\|f^{\prime}\left(x_{k}\right)\right\|_{2,\infty}^{2} m} $  \cr
            MD-SNK  &   $1-  \frac{1-2\eta}{1+\eta^2} \frac{ h_2^2(G \left(x_k\right))}{m}$                          & $1-  \frac{1-2\eta}{1+\eta^2}  \frac{\beta h_2^2(G \left(x_k\right))}{m}$  \cr
            MD-NK   &               $1-  \frac{1-2\eta}{1+\eta^2}\frac{h_2^2(G \left(x_k\right))}{m}    $              &  $1-  \frac{1-2\eta}{1+\eta^2}    h_2^2(G \left(x_k\right)) $      \cr
            NURK-MD &    $1-  \frac{1-2\eta}{1+\eta^2}  \frac{ h_2^2(G \left(x_k\right))}{m}$                         &   $1-  \frac{1-2\eta}{1+\eta^2}  \frac{ h_2^2(G \left(x_k\right))}{m}$     \cr\hline
\end{tabular}
\end{table}
\begin{table}[]
\centering
   \fontsize{8}{8}\selectfont
       \caption{Comparison of the best case convergence factors listed in \Cref{Sec4:tab1}.}
    \label{Sec4:tab2}
    \begin{tabular}{ c| c c c c c c  }
 \hline
             &   MD-NK  &  MR-NK & MD-SNK & MR-SNK & NURK-MD & NURK-MR                    \cr \hline
   MD-NK             &  $=$    &  $\leq$  & $\leq$ & $\leq$  &$\leq$  & $\leq$            \cr
   MR-NK             &         &   $=$    &  P     & $\leq$  & P      &$\leq$             \cr
   MD-SNK            &         &          &$=$     & $\leq$  & $\leq$ & $\leq$            \cr
   MR-SNK            &         &          &        & $=$     &  P     & $\leq$            \cr
   NURK-MD           &         &          &        &         & $=$    &  $\leq$            \cr
   NURK-MR           &         &          &        &         &        & $=$                \cr\hline
\end{tabular}
\end{table}
\end{remark}

\section{Multiple samples-based greedy randomized sampling methods}
\label{sec: mutiple-methods}
In this section, on the basis of the MR/MD-SNK method, we further consider block version nonlinear iteration and propose block sampling nonlinear Kaczmarz (BSNK) method for solving the nonlinear problems (\ref{Sec11}). 
Rather than producing the next iteration to satisfy the single sampled equation as in the MR/MD-SNK method, block iteration satisfies all the equations in the sampled multiple samples. Thus, given  an index subset of samples $\tau\subseteq [m]$, the projection step of the BSNK method is obtained by
\begin{align}
x_{k+1}=\mathop{\text{argmin}}\limits_{x\in \mathbb{R}^{n}}\|x-x_k\|_2^2,\quad s.t. \quad f_{\tau}(x_k)+  f^{\prime}_{\tau}(x_k) (x-x_k)=0.\notag
\end{align}
Thus, the solution of the projection is
\begin{align}
x_{k+1}=x_{k}-(f_{\tau}^{\prime}(x_k))^{\dagger}f_{\tau}(x_k).  \label{Sec5-1}
\end{align}

From the update formula (\ref{Sec5-1}) of the block method, we know that the choice of the index subset $\tau$ greatly affects the performance of the algorithm. So, inspired by the work of Zhang and Li \cite{zhang2021block}, we construct two algorithmic frameworks for determining the index subset, which combining with the greedy rules presented in \Cref{subsec: rules} lead to four greedy BSNK methods shown in \Cref{alg-The (MR/MD)-BSNK1 method,alg-The (MR/MD)-BSNK2 method} for the nonlinear problems (\ref{Sec11}). \Cref{alg-The (MR/MD)-BSNK1 method,alg-The (MR/MD)-BSNK2 method} indicate that their iterations make faster progress than that of \Cref{alg2}. The main reason is that the iteration index $i_k$ in \Cref{alg2} belongs to the index set used in \Cref{alg-The (MR/MD)-BSNK1 method,alg-The (MR/MD)-BSNK2 method}. In addition, if $\beta=m$ in \Cref{alg-The (MR/MD)-BSNK1 method} and $\nu=1$ in \Cref{alg-The (MR/MD)-BSNK2 method}, 
then $|\mathcal{I}_{k}|=|\mathcal{J}_{k}|=1$ and thus the (MR/MD)-BSNK1 and (MR/MD)-BSNK2 methods will recover the (MR/MD)-NK method.

\begin{algorithm}
\caption{The (MR/MD)-BSNK1 method }
\label{alg-The (MR/MD)-BSNK1 method}
\begin{algorithmic}[1]
\STATE{Input: The initial estimate $x_0\in \mathbb{R}^n$, parameter $\beta \in [m]$. }
\FOR{$k=0, 1, 2, \cdots $ until convergence,}
\STATE{Choose an index subset $\tau_k$ of size $\beta$ uniformly at random from among $[ m ]$.}
\STATE{Switch the greedy sampling strategy to determine two different iterative methods, i.e., the MR-BSNK1 and MD-BSNK1 methods.}
\STATE{\quad $\triangleright$ case 1:  maximum residual (MR) rule
\\
\quad \quad $i_k={\rm arg} \max \limits _{i\in \tau_k} \left|f_{i}(x_k)\right|^2  $ and $\delta_k=\max \limits _{i\in \tau_k}  \left|f_{i}(x_k)\right|^2$.
\\
\quad \quad Determine the index set
$  \mathcal{I}_{k}=\{h_k|\left|f_{h_k}(x_k)\right|^2\geq \delta_k;~h_k\in[m]\backslash\tau_k\}\cup\{i_k\}$.
}
\STATE{\quad $\triangleright$ case 2: maximum distance (MD) rule
\\
\quad\quad $i_k={\rm arg} \max \limits _{i\in \tau_k} \frac{\left|f_{i}(x_k)\right|^{2}}{\left\|\nabla f_{i}\left(x_{k}\right)\right\|_2^{2}}$ and $\delta_k=\max \limits _{i\in \tau_k}  \frac{\left|f_{i}(x_k)\right|^{2}}{\left\|\nabla f_{i}\left(x_{k}\right)\right\|_2^{2}} $.
\\
\quad\quad Determine the index set
$  \mathcal{I}_{k}=\{h_k |  \frac{\left|f_{h_k}(x_k)\right|^{2}}{\left\|\nabla f_{h_k}\left(x_{k}\right)\right\|_2^{2}} \geq \delta_k;~h_k\in[ m]\backslash\tau_k\}\cup\{i_k\}$.}
\STATE{Update $x_{k+1}=x_{k}-(f_{ \mathcal{I}_{k}}^{\prime}(x))^{\dagger}f_{ \mathcal{I}_{k}}(x_k)$. }
\ENDFOR
\end{algorithmic}
\end{algorithm}

\begin{algorithm}
\caption{The (MR/MD)-BSNK2 method}
\label{alg-The (MR/MD)-BSNK2 method}
\begin{algorithmic}[1]
\STATE{Input: The initial estimate $x_0\in \mathbb{R}^n$, parameter $\nu\in[m]$, parameters $\beta_1, \beta_2, \cdots, \beta_{\nu}$ s.t. $\sum \limits _{j=1}^{\nu}\beta_{j}=m$. }
\FOR{$j=1: \nu$ }
\STATE{Choose an index subset $\tau_j$ of size $\beta_j$ uniformly at random from among the remaining indices of $[ m ]$ without replacement.  }
\STATE{Switch the greedy sampling strategy to determine two different iterative methods, i.e., the MR-BSNK2 and MD-BSNK2 methods.}
\STATE{\quad $\triangleright$ case 1:  maximum residual (MR) rule
\\
\quad\quad $t_j={\rm arg} \max \limits _{i\in \tau_j} \left|f_{i}(x_k)\right|^2 . $}
\STATE{\quad $\triangleright$ case 2: maximum distance (MD) rule
\\
 \quad\quad $t_j={\rm arg} \max \limits _{i\in \tau_j} \frac{\left|f_{i}(x_k)\right|^{2}}{\left\|\nabla f_{i}\left(x_{k}\right)\right\|_2^{2}} .$ }
 \ENDFOR
\STATE{Determine the index set
$
  \mathcal{J}_{k}=\{t_1, t_2, \cdots, t_{\nu}\}
$.}
\STATE{Update $x_{k+1}=x_{k}-(f_{ \mathcal{J}_{k}}^{\prime}(x))^{\dagger}f_{ \mathcal{J}_{k}}(x_k)$. }

\end{algorithmic}
\end{algorithm}

\section{Convergence analysis}
\label{sec:theory-mutiple}
We now present the convergence results for the four block methods proposed in \Cref{sec: mutiple-methods}, i.e., the MR-BSNK1, MD-BSNK1, MR-BSNK2 and MD-BSNK2 methods.
\begin{lemma}  \label{sec:theory-mutiple-lemma1}
If the nonlinear function $f $ satisfies the local tangential cone condition given in \Cref{definition} and a vector $x_\star\in \mathbb{R}^{n} $ satisfies $f(x_\star)=0$, then from the iteration $x_{k+1}=x_{k}-(f_{\tau_k}^{\prime}(x_k))^{\dagger}f_{\tau_k}(x_k)$ with $\tau_k\subseteq[m]$, we have
\begin{align}
\left\|x_{k+1}-x_{\star}\right\|^{2}_2
\leq
\left\|x_{k}-x_{\star}\right\|^{2}_2-  \left(h_2^2((f_{\tau_k}^{\prime}(x_k))^{\dagger})-2\eta\sigma^2_{\max}\left((f_{\tau_k}^{\prime}(x_k))^{\dagger}\right)\right)\left\|  f_{\tau_k}(x_k)\right\|^{2}_2. \notag
\end{align}
\end{lemma}

\begin{proof}
From the iteration formula $x_{k+1}=x_{k}-(f_{\tau_k}^{\prime}(x_k))^{\dagger}f_{\tau_k}(x_k)$, we have
\begin{align}
\left\|x_{k+1}-x_{\star}\right\|^{2}_2
&=
 \left\|x_{k}-(f_{\tau_k}^{\prime}(x_k))^{\dagger}f_{\tau_k}(x_k)-x_{\star}\right\|^{2}_2  \notag
\\
&= \left\|x_{k}-x_{\star}\right\|^{2}_2 + \left\| (f_{\tau_k}^{\prime}(x_k))^{\dagger}f_{\tau_k}(x_k)\right\|^{2}_2 - 2<(f_{\tau_k}^{\prime}(x_k))^{\dagger}f_{\tau_k}(x_k), x_{k}-x_{\star}>, \notag
\end{align}
which together with the fact that $(f_{\tau_k}^{\prime}(x_k))^{\dagger}=(f_{\tau_k}^{\prime}(x_k))^T(f_{\tau_k}^{\prime}(x_k)(f_{\tau_k}^{\prime}(x_k))^T)^{\dagger}$ leads to
\begin{align}
\left\|x_{k+1}-x_{\star}\right\|^{2}_2
&=
\left\|x_{k}-x_{\star}\right\|^{2}_2 + \left\| (f_{\tau_k}^{\prime}(x_k))^{\dagger}f_{\tau_k}(x_k)\right\|^{2}_2 - 2 f^T_{\tau_k}(x_k)(f_{\tau_k}^{\prime}(x_k)(f_{\tau_k}^{\prime}(x_k))^T)^{\dagger}f_{\tau_k}^{\prime}(x_k)( x_{k}-x_{\star}) \notag
\\
&=
\left\|x_{k}-x_{\star}\right\|^{2}_2+ \left\| (f_{\tau_k}^{\prime}(x_k))^{\dagger}f_{\tau_k}(x_k)\right\|^{2}_2 \notag
\\
& \quad + 2f^T_{\tau_k}(x_k)(f_{\tau_k}^{\prime}(x_k)(f_{\tau_k}^{\prime}(x_k))^T)^{\dagger}\left(f_{\tau_k}(x_k)-
f_{\tau_k}(x_{\star})-f_{\tau_k}^{\prime}(x_k)( x_{k}-x_{\star})\right) \notag
\\
& \quad -2f^T_{\tau_k}(x_k)(f_{\tau_k}^{\prime}(x_k)(f_{\tau_k}^{\prime}(x_k))^T)^{\dagger} f_{\tau_k}(x_k). \notag
\end{align}
Since
\begin{align}\left\| (f_{\tau_k}^{\prime}(x_k))^{\dagger}f_{\tau_k}(x_k)\right\|^{2}_2
&=f^T_{\tau_k}(x_k) (f_{\tau_k}^{\prime}(x_k))^{\dagger})^Tf_{\tau_k}^{\prime}(x_k))^{\dagger}f_{\tau_k}(x_k)\notag
\\
&=f^T_{\tau_k}(x_k) (f_{\tau_k}^{\prime}(x_k))^T)^{\dagger}f_{\tau_k}^{\prime}(x_k))^{\dagger}f_{\tau_k}(x_k)\notag
\\
&=
f^T_{\tau_k}(x_k)(f_{\tau_k}^{\prime}(x_k)(f_{\tau_k}^{\prime}(x_k))^T)^{\dagger} f_{\tau_k}(x_k),\notag
\end{align}
we get
\begin{align}
\left\|x_{k+1}-x_{\star}\right\|^{2}_2
&=
\left\|x_{k}-x_{\star}\right\|^{2}_2- \left\| (f_{\tau_k}^{\prime}(x_k))^{\dagger}f_{\tau_k}(x_k)\right\|^{2}_2 \notag
\\
& \quad + 2f^T_{\tau_k}(x_k)(f_{\tau_k}^{\prime}(x_k)(f_{\tau_k}^{\prime}(x_k))^T)^{\dagger}\left(f_{\tau_k}(x_k)-
f_{\tau_k}(x_{\star})-f_{\tau_k}^{\prime}(x_k)( x_{k}-x_{\star})\right). \notag
\end{align}
Furthermore, using Cauchy-Schwarz inequality, it is true that
\begin{align}
\left\|x_{k+1}-x_{\star}\right\|^{2}_2
&\leq
\left\|x_{k}-x_{\star}\right\|^{2}_2- \left\| (f_{\tau_k}^{\prime}(x_k))^{\dagger}f_{\tau_k}(x_k)\right\|^{2}_2 \notag
\\
& \quad + 2\|f^T_{\tau_k}(x_k)(f_{\tau_k}^{\prime}(x_k)(f_{\tau_k}^{\prime}(x_k))^T)^{\dagger}\|_2\| f_{\tau_k}(x_k)-
f_{\tau_k}(x_{\star})-f_{\tau_k}^{\prime}(x_k)( x_{k}-x_{\star}) \|_2\notag
\\
&\leq
\left\|x_{k}-x_{\star}\right\|^{2}_2-  h_2^2((f_{\tau_k}^{\prime}(x_k))^{\dagger})\left\|  f_{\tau_k}(x_k)\right\|^{2}_2 \notag
\\
& \quad + 2\sigma^2_{\max}\left((f_{\tau_k}^{\prime}(x_k))^{\dagger}\right)\|f_{\tau_k}(x_k) \|_2\| f_{\tau_k}(x_k)-
f_{\tau_k}(x_{\star})-f_{\tau_k}^{\prime}(x_k)( x_{k}-x_{\star}) \|_2.\notag
\end{align}
Finally, we obtain the desired result by applying (\ref{Sec2.6}) in \Cref{lemma3}.

\end{proof}

\begin{theorem}
\label{theorem:MR-BSNK1}
If the nonlinear function $f $ satisfies the local tangential cone condition given in \Cref{definition}, $\eta=\max\limits _{i}
\eta_{i}<\frac{1}{2}$,
\begin{align}
\varepsilon=h_2^2((f_{\mathcal{I} }^{\prime}(x_k))^{\dagger})-2\eta\sigma^2_{\max}\left((f_{\mathcal{I} }^{\prime}(x_k))^{\dagger}\right)>0,\label{theorem:MR-BSNK1-1}
\end{align}
where $h_2^2((f_{\mathcal{I}}^{\prime}(x_k))^{\dagger})= \min \limits _{\mathcal{I}_{k}} h_2^2((f_{\mathcal{I}_{k}}^{\prime}(x_k))^{\dagger})$ and $\sigma^2_{\max}\left((f_{\mathcal{I} }^{\prime}(x_k))^{\dagger}\right)= \max \limits _{\mathcal{I}_{k}}\sigma^2_{\max}\left((f_{\mathcal{I}_{k}}^{\prime}(x_k))^{\dagger}\right)$, and $f(x_{\star})=0$, then the iterations of the MR-BSNK1 method, i.e., the first case in \Cref{alg-The (MR/MD)-BSNK1 method}, satisfy
\begin{align}
\mathbb{E}\left[\left\| {x}_{k+1}- {x}_{\star}\right\|^{2}_2\right]
 \leq
\left(1 - \varepsilon |\mathcal{I} |  \frac{1}{\xi_k} \frac{\beta}{m}\frac{1}{1+\eta^2}h_2^2\left( f^{\prime}\left(x_{k}\right)\right )\right )
 \mathbb{E}\left[\left\|x_{k}-{x}_{\star}\right\|^{2}_2\right],    \label{theorem:MR-BSNK1-2}
 \end{align}
  where
  \begin{align}
|\mathcal{I} |= \min \limits _{\mathcal{I}_{k}}|\mathcal{I}_{k}|
\quad\text{and}\quad
\xi_k=\frac{\sum  \limits _{\tau_k \in \binom{[m]}{\beta}} \left \|f_{\tau_k}(x_k)\right\|_{2}^{2}}{\sum  \limits _{\tau_k \in \binom{[m]}{\beta}}  \left\|f_{\tau_k}(x_k)\right\|_{\infty}^{2}}.\label{theorem:MR-BSNK1-3}
\end{align}
\end{theorem}

\begin{proof}
According to \Cref{sec:theory-mutiple-lemma1}, (\ref{theorem:MR-BSNK1-1}) and \Cref{alg-The (MR/MD)-BSNK1 method}, we have
\begin{align}
\left\|x_{k+1}-x_{\star}\right\|^{2}_2
&\leq
\left\|x_{k}-x_{\star}\right\|^{2}_2-  \left(h_2^2((f_{\mathcal{I}_{k}}^{\prime}(x_k))^{\dagger})-2\eta\sigma^2_{\max}\left((f_{\mathcal{I}_{k}}^{\prime}(x_k))^{\dagger}\right)\right)\left\|  f_{\mathcal{I}_{k}}(x_k)\right\|^{2}_2  \notag
\\
&\leq
\left\|x_{k}-x_{\star}\right\|^{2}_2-  \varepsilon\left\|  f_{\mathcal{I}_{k}}(x_k)\right\|^{2}_2 \notag
\\
&=
\left\|x_{k}-x_{\star}\right\|^{2}_2-  \varepsilon \sum\limits_{j\in\mathcal{I}_{k} }\left|  f_{j}(x_k)\right|^{2} \notag
\\
&\leq
\left\|x_{k}-x_{\star}\right\|^{2}_2-  \varepsilon |\mathcal{I}_{k}|\delta_k   \notag
\\
&=
\left\|x_{k}-x_{\star}\right\|^{2}_2-  \varepsilon |\mathcal{I}_{k}| \left\|  f_{\tau_k}(x_k)\right\|^{2}_{\infty}. \notag
\end{align}
Taking the expectation on both sides conditioned on $x_{k}$ and combining with the definitions of $|\mathcal{I} |$ and $\xi_k$ in (\ref{theorem:MR-BSNK1-3}), we get
\begin{align}
\mathbb{E}^k\left[\left\|x_{k+1}-x_{\star}\right\|^{2}_2\right]
&\leq
\left\|x_{k}-x_{\star}\right\|^{2}_2-  \varepsilon\mathbb{E}^k\left[ |\mathcal{I}_{k}| \left\|  f_{\tau_k}(x_k)\right\|^{2}_{\infty}\right] \notag
\\
&=
\left\|x_{k}-x_{\star}\right\|^{2}_2- \varepsilon \sum\limits_{\tau_k\in\binom{[m]}{\beta}}\frac{1}{\binom{m}{\beta}} |\mathcal{I}_{k}| \left\|  f_{\tau_k}(x_k)\right\|^{2}_{\infty}  \notag
\\
&\leq
\left\|x_{k}-x_{\star}\right\|^{2}_2-  \varepsilon |\mathcal{I} |\sum\limits_{\tau_k\in\binom{[m]}{\beta}}\frac{1}{\binom{m}{\beta}} \left\|  f_{\tau_k}(x_k)\right\|^{2}_{\infty}  \notag
\\
&=
\left\|x_{k}-x_{\star}\right\|^{2}_2-  \varepsilon |\mathcal{I} | \frac{1}{\binom{m}{\beta}} \frac{1}{\xi_k} \sum\limits_{\tau_k\in\binom{[m]}{\beta}} \left\|  f_{\tau_k}(x_k)\right\|^{2}_{2}  \notag
\\
&=
\left\|x_{k}-x_{\star}\right\|^{2}_2-  \varepsilon |\mathcal{I} | \frac{1}{\binom{m}{\beta}} \frac{1}{\xi_k} \frac{\binom{m}{\beta}\beta}{m}\left\|f(x_k)\right\|^{2}_{2}  \notag
\\
&=
\left\|x_{k}-x_{\star}\right\|^{2}_2-  \varepsilon |\mathcal{I} |  \frac{1}{\xi_k} \frac{\beta}{m}\left\|f(x_k)-f(x_{\star})\right\|^{2}_{2}.  \notag
\end{align}
Moreover, noting that (\ref{Sec2.7}) in \Cref{lemma3} and the definition of $h_2(\cdot)$ in \Cref{subsec:notation}, we get
\begin{align}
\mathbb{E}^k\left[\left\|x_{k+1}-x_{\star}\right\|^{2}_2\right]
&\leq
\left\|x_{k}-x_{\star}\right\|^{2}_2-  \varepsilon |\mathcal{I} |  \frac{1}{\xi_k} \frac{\beta}{m}\frac{1}{1+\eta^2}\left\|f^{\prime}\left(x_{k}\right)\left(x_{k}-x_{\star}\right)\right\|^{2}_{2}  \notag
\\
&\leq
\left\|x_{k}-x_{\star}\right\|^{2}_2-  \varepsilon |\mathcal{I} |  \frac{1}{\xi_k} \frac{\beta}{m}\frac{1}{1+\eta^2}h_2^2\left( f^{\prime}\left(x_{k}\right)\right )\left\| x_{k}-x_{\star}\right\|^{2}_{2} \notag
\\
&\leq
\left(1 - \varepsilon |\mathcal{I} |  \frac{1}{\xi_k} \frac{\beta}{m}\frac{1}{1+\eta^2}h_2^2\left( f^{\prime}\left(x_{k}\right)\right )\right )\left\| x_{k}-x_{\star}\right\|^{2}_{2}.\notag
\end{align}
Thus, taking expectation on both sides and using the tower rule of expectation, we get the desired result (\ref{theorem:MR-BSNK1-2}).
\end{proof}

\begin{remark}
\label{remark_th_MR-BSNK1-1}
Since the iterations of the MR-BSNK1 method satisfy
\begin{align}
\mathbb{E}\left[\left\| {x}_{k+1}- {x}_{\star}\right\|^{2}_2\right]
 &\leq
\left(1 - \varepsilon |\mathcal{I} |  \frac{1}{\xi_k} \frac{\beta}{m}\frac{1}{1+\eta^2}h_2^2\left( f^{\prime}\left(x_{k}\right)\right )\right )
 \mathbb{E}\left[\left\|x_{k}-{x}_{\star}\right\|^{2}_2\right]\notag
 \\
  &=
 \mathbb{E}\left[\left\|x_{k}-{x}_{\star}\right\|^{2}_2\right]  - \varepsilon |\mathcal{I} |  \frac{1}{\xi_k} \frac{\beta}{m}\frac{1}{1+\eta^2}h_2^2\left( f^{\prime}\left(x_{k}\right)\right )
 \mathbb{E}\left[\left\|x_{k}-{x}_{\star}\right\|^{2}_2\right],\notag
 \end{align}
 which together with the facts $0\leq \mathbb{E}\left[\left\| {x}_{k+1}- {x}_{\star}\right\|^{2}_2\right]$ and $\varepsilon |\mathcal{I} |  \frac{1}{\xi_k} \frac{\beta}{m}\frac{1}{1+\eta^2}h_2^2\left( f^{\prime}\left(x_{k}\right)\right ) >0$ indicates
\begin{align}
0
\leq
 \mathbb{E}\left[\left\|x_{k}-{x}_{\star}\right\|^{2}_2\right]  - \varepsilon |\mathcal{I} |  \frac{1}{\xi_k} \frac{\beta}{m}\frac{1}{1+\eta^2}h_2^2\left( f^{\prime}\left(x_{k}\right)\right )
 \mathbb{E}\left[\left\|x_{k}-{x}_{\star}\right\|^{2}_2\right]
 <\mathbb{E}\left[\left\|x_{k}-{x}_{\star}\right\|^{2}_2\right]  .\notag
 \end{align}
Then we can obtain that the convergence factor of the MR-BSNK1 method is smaller than 1. Similarly, we can get that the convergence factors of the MD-BSNK1, MR-BSNK2 and MD-BSNK2 methods, which are respectively presented in \Cref{theorem:MD-BSNK1,theorem:MR-BSNK2,theorem:MD-BSNK2}, are also smaller than 1.
\end{remark}

\begin{theorem}
\label{theorem:MD-BSNK1}
If the nonlinear function $f $ satisfies the local tangential cone condition given in \Cref{definition}, $\eta=\max\limits _{i}
\eta_{i}<\frac{1}{2}$, $
\varepsilon=h_2^2((f_{\mathcal{I} }^{\prime}(x_k))^{\dagger})-2\eta\sigma^2_{\max}\left((f_{\mathcal{I} }^{\prime}(x_k))^{\dagger}\right)>0,$
where $h_2^2((f_{\mathcal{I}}^{\prime}(x_k))^{\dagger})= \min \limits _{\mathcal{I}_{k}} h_2^2((f_{\mathcal{I}_{k}}^{\prime}(x_k))^{\dagger})$ and $\sigma^2_{\max}\left((f_{\mathcal{I} }^{\prime}(x_k))^{\dagger}\right)= \max \limits _{\mathcal{I}_{k}}\sigma^2_{\max}\left((f_{\mathcal{I}_{k}}^{\prime}(x_k))^{\dagger}\right)$, and $f(x_{\star})=0$, then the iterations of MD-BSNK1 method, i.e., the second case in \Cref{alg-The (MR/MD)-BSNK1 method}, satisfy
\begin{align}
\mathbb{E}\left[\left\| {x}_{k+1}- {x}_{\star}\right\|^{2}_2\right]
 \leq
\left(1-   \varepsilon  \min \limits _{i\in[m]} \left\|\nabla f_{i}\left(x_{k}\right)\right\|_2^{2}|\mathcal{I}|\frac{\beta}{\zeta_k}
\frac{1}{1+\eta^2} \frac{ h_2^2(G \left(x_k\right))}{m}\right)
 \mathbb{E}\left[\left\|x_{k}-{x}_{\star}\right\|^{2}_2\right],\label{theorem:MD-BSNK1-2}
 \end{align}
  where
  \begin{align}
|\mathcal{I} |= \min \limits _{\mathcal{I}_{k}}|\mathcal{I}_{k}|
\quad\text{and}\quad
\zeta_k=\frac{\sum  \limits _{\tau_k \in \binom{[m]}{\beta}} \left \|u_{\tau_k}(x_k)\right\|_{2}^{2}}{\sum  \limits _{\tau_k \in \binom{[m]}{\beta}}  \left\|u_{\tau_k}(x_k)\right\|_{\infty}^{2}}.\label{theorem:MD-BSNK1-3}
\end{align}
\end{theorem}

\begin{proof}
Similar to the proof of \Cref{theorem:MR-BSNK1}, we can obtain
\begin{align}
\left\|x_{k+1}-x_{\star}\right\|^{2}_2
&\leq
\left\|x_{k}-x_{\star}\right\|^{2}_2-  \varepsilon \sum\limits_{j\in\mathcal{I}_{k} }\left|  f_{j}(x_k)\right|^{2} \notag
\\
&=
\left\|x_{k}-x_{\star}\right\|^{2}_2-  \varepsilon \sum\limits_{j\in\mathcal{I}_{k} }\frac{\left|  f_{j}(x_k)\right|^{2}}{\left\|\nabla f_{j}\left(x_{k}\right)\right\|_2^{2}} \left\|\nabla f_{j}\left(x_{k}\right)\right\|_2^{2}\notag
\\
&\leq
\left\|x_{k}-x_{\star}\right\|^{2}_2-  \varepsilon  \min \limits _{i\in[m]} \left\|\nabla f_{i}\left(x_{k}\right)\right\|_2^{2} \sum\limits_{j\in\mathcal{I}_{k} }\frac{\left|  f_{j}(x_k)\right|^{2}}{\left\|\nabla f_{j}\left(x_{k}\right)\right\|_2^{2}},\notag
\end{align}
which together with the definition of $u(x)$ in \Cref{subsec:notation} and \Cref{alg-The (MR/MD)-BSNK1 method} leads to
\begin{align}
\left\|x_{k+1}-x_{\star}\right\|^{2}_2
&\leq
\left\|x_{k}-x_{\star}\right\|^{2}_2-  \varepsilon  \min \limits _{i\in[m]} \left\|\nabla f_{i}\left(x_{k}\right)\right\|_2^{2} \sum\limits_{j\in\mathcal{I}_{k} } \left|u_{j}(x_k)\right|^{2}  \notag
\\
&\leq
\left\|x_{k}-x_{\star}\right\|^{2}_2-  \varepsilon  \min \limits _{i\in[m]} \left\|\nabla f_{i}\left(x_{k}\right)\right\|_2^{2} |\mathcal{I}_{k}| \delta_k  \notag
\\
&=
\left\|x_{k}-x_{\star}\right\|^{2}_2-  \varepsilon  \min \limits _{i\in[m]} \left\|\nabla f_{i}\left(x_{k}\right)\right\|_2^{2} |\mathcal{I}_{k}| \left\|u_{\tau_k}(x_k)\right\|^{2}_\infty  \notag
\end{align}
Now, taking expectation of both sides conditioned on $x_{k}$ gives
\begin{align}
\mathbb{E}^{k}\left[\left\| {x}_{k+1}- {x}_{\star}\right\|^{2}_2\right]
&\leq \left\|x_{k}-x_{\star}\right\|^{2}_2 -\mathbb{E}^{k}\left[ \varepsilon  \min \limits _{i\in[m]} \left\|\nabla f_{i}\left(x_{k}\right)\right\|_2^{2} |\mathcal{I}_{k}| \left\|u_{\tau_k}(x_k)\right\|^{2}_\infty\right] \notag
\\
&=\left\|x_{k}-x_{\star}\right\|^{2}_2 - \sum\limits_{\tau_k\in\binom{[m]}{\beta}}\frac{1}{\binom{m}{\beta}}\varepsilon  \min \limits _{i\in[m]} \left\|\nabla f_{i}\left(x_{k}\right)\right\|_2^{2} |\mathcal{I}_{k}|
\left\|u_{\tau_k}(x_k)\right\|^{2}_\infty\notag
\\
&\leq \left\|x_{k}-x_{\star}\right\|^{2}_2 -\frac{1}{\binom{m}{\beta}}\varepsilon  \min \limits _{i\in[m]} \left\|\nabla f_{i}\left(x_{k}\right)\right\|_2^{2}\sum\limits_{\tau_k\in\binom{[m]}{\beta}}|\mathcal{I}_{k}|
\left\|u_{\tau_k}(x_k)\right\|^{2}_\infty ,   \notag
\end{align}
which together with the definitions of $|\mathcal{I} |$ and $\zeta_k$ in (\ref{theorem:MD-BSNK1-3}), we have
\begin{align}
\mathbb{E}^{k}\left[\left\| {x}_{k+1}- {x}_{\star}\right\|^{2}_2\right]
&\leq
\left\|x_{k}-x_{\star}\right\|^{2}_2 -\frac{1}{\binom{m}{\beta}}\varepsilon  \min \limits _{i\in[m]} \left\|\nabla f_{i}\left(x_{k}\right)\right\|_2^{2}|\mathcal{I} |\frac{1}{\zeta_k} \sum\limits_{\tau_k\in\binom{[m]}{\beta}}
\left\|u_{\tau_k}(x_k)\right\|^{2}_2   \notag
\\
&=
\left\|x_{k}-x_{\star}\right\|^{2}_2 -\frac{1}{\binom{m}{\beta}}\varepsilon  \min \limits _{i\in[m]} \left\|\nabla f_{i}\left(x_{k}\right)\right\|_2^{2}|\mathcal{I}| \frac{1}{\zeta_k} \frac{\binom{m}{\beta}\beta}{m}
\left\|u(x_k)\right\|^{2}_2   \notag
\\
&=
\left\|x_{k}-x_{\star}\right\|^{2}_2 - \varepsilon  \min \limits _{i\in[m]} \left\|\nabla f_{i}\left(x_{k}\right)\right\|_2^{2}|\mathcal{I}|\frac{\beta}{\zeta_k} \frac{ 1}{m}
\left\|u(x_k)\right\|^{2}_2  . \notag
\end{align}
Moreover, by making use of (\ref{Sec2.7.1}) in \Cref{lemma4} and the definition of $h_2(\cdot)$ in \Cref{subsec:notation}, we can obtain
\begin{align}
\mathbb{E}^{k}\left[\left\| {x}_{k+1}- {x}_{\star}\right\|^{2}_2\right]
&\leq
\left\|x_{k}-x_{\star}\right\|^{2}_2 -  \varepsilon  \min \limits _{i\in[m]} \left\|\nabla f_{i}\left(x_{k}\right)\right\|_2^{2}|\mathcal{I}|\frac{\beta}{\zeta_k} \frac{ 1}{m}
\frac{1}{1+\eta^2}\left\|G \left(x_k\right)\left(x_k-x_{\star}\right)\right\|^2_2  \notag
\\
&\leq
\left(1-   \varepsilon  \min \limits _{i\in[m]} \left\|\nabla f_{i}\left(x_{k}\right)\right\|_2^{2}|\mathcal{I}|\frac{\beta}{\zeta_k}
\frac{1}{1+\eta^2} \frac{ h_2^2(G \left(x_k\right))}{m}\right)
\left\|\left(x_k-x_{\star}\right)\right\|^2_2 . \notag
\end{align}
Finally, the desired result (\ref{theorem:MD-BSNK1-2}) can be obtained by taking the full expectation on both sides.

\end{proof}

\begin{remark}\label{remark_th_MD-BSNK1-1}
Letting $D$ be an $m$-by-$m$ diagonal matrix of the following form
 $$
 D=\begin{bmatrix}
 \|\nabla f_1(x_k)\|_2&&\\
 &\ddots&\\
 && \|\nabla f_m(x_k)\|_2
 \end{bmatrix},
 $$
then we get a relation between $f^{\prime}(x_k)$ and $G(x_k)$, that is,
\begin{align}
f^{\prime}(x_k)=DG(x_k).\label{remark_theorem_MD-BSNK1:1}
 \end{align}
Now, by setting $A=G(x_k)G(x_k)^T$ and $B=D^TD$, since both matrices $A$ and $B$ are symmetric positive-semidefinite, we have
\begin{align}
\lambda_1(B)I\geq B\geq \lambda_m(B)I\notag
 \end{align}
and then obtain
\begin{align}
A^{\frac{1}{2}}BA^{\frac{1}{2}}= A^{\frac{1}{2}}(B-\lambda_m(B)I)A^{\frac{1}{2}}+ \lambda_m(B) A\geq  \lambda_m(B) A.\notag
 \end{align}
Further, by applying \Cref{lemma2}, we get
\begin{align}
\lambda_m\left(A^{\frac{1}{2}}BA^{\frac{1}{2}}\right)
&=\lambda_m\left(A^{\frac{1}{2}}(B-\lambda_m(B)I)A^{\frac{1}{2}}+ \lambda_m(B) A\right)\notag
\\
&\geq\lambda_m\left(A^{\frac{1}{2}}(B-\lambda_m(B)I)A^{\frac{1}{2}}  \right)+\lambda_m\left( \lambda_m(B) A\right)\notag
\\
&=\lambda_m\left(B\right)\lambda_m\left(A\right),\notag
 \end{align}
which implies
\begin{align}
  \lambda_m\left(AB\right)\geq\lambda_m(B)\lambda_m\left(A\right).\notag
 \end{align}
 That is,
 \begin{align}
\lambda_m\left(G(x_k)G(x_k)^TD^TD\right)=\lambda_m\left(DG(x_k)G(x_k)^TD^T\right)
\geq \lambda_m(D^TD)\lambda_m\left(G(x_k)G(x_k)^T\right),\notag
 \end{align}
 which together with the definitions of $D$, $G(x_k)$ and $f^{\prime}(x_k)=DG(x_k)$ in (\ref{remark_theorem_MD-BSNK1:1}) indicates
  \begin{align}
h_2^2\left( f^{\prime}\left(x_{k}\right)\right )\geq \min \limits _{i\in[m]} \left\|\nabla f_{i}\left(x_{k}\right)\right\|_2^{2} h_2^2(G \left(x_k\right))
 .\label{remark_theorem_MD-BSNK1:2}
 \end{align}
Meanwhile, from \Cref{theorem:MR-BSNK1,theorem:MD-BSNK1}, we know that the convergence factors of the MR-BSNK1 and MD-BSNK1 methods are
$$
\rho_{\text{MR-BSNK1}}\sim 1 -     \frac{\varepsilon}{m}\frac{ |\mathcal{I} |}{1+\eta^2}h_2^2\left( f^{\prime}\left(x_{k}\right)\right ),
\quad
\text{and}
\quad
\rho_{\text{MD-BSNK1}}\sim 1-   \frac{\varepsilon}{m}
\frac{|\mathcal{I}|}{1+\eta^2}  \min \limits _{i\in[m]} \left\|\nabla f_{i}\left(x_{k}\right)\right\|_2^{2} h_2^2(G \left(x_k\right)) ,
$$
respectively.
Thus, according to (\ref{remark_theorem_MD-BSNK1:2}), we conclude that
\begin{align}
1 -     \frac{\varepsilon}{m}\frac{ |\mathcal{I} |}{1+\eta^2}h_2^2\left( f^{\prime}\left(x_{k}\right)\right )
\leq
   1-   \frac{\varepsilon}{m}
\frac{|\mathcal{I}|}{1+\eta^2}  \min \limits _{i\in[m]} \left\|\nabla f_{i}\left(x_{k}\right)\right\|_2^{2} h_2^2(G \left(x_k\right)), \notag
\end{align}
which means that the convergence factor of the MR-BSNK1 method is a smaller than that of the MD-BSNK1 method.
\end{remark}

\begin{theorem}
\label{theorem:MR-BSNK2}
If the nonlinear function $f $ satisfies the local tangential cone condition given in \Cref{definition}, $\eta=\max\limits _{i}
\eta_{i}<\frac{1}{2}$, $
\alpha=h_2^2((f_{\mathcal{J} }^{\prime}(x_k))^{\dagger})-2\eta\sigma^2_{\max}\left((f_{\mathcal{J} }^{\prime}(x_k))^{\dagger}\right)>0,$
where $h_2^2((f_{\mathcal{J}}^{\prime}(x_k))^{\dagger})= \min \limits _{\mathcal{J}_{k}} h_2^2((f_{\mathcal{J}_{k}}^{\prime}(x_k))^{\dagger})$ and $\sigma^2_{\max}\left((f_{\mathcal{J} }^{\prime}(x_k))^{\dagger}\right)= \max \limits _{\mathcal{J}_{k}}\sigma^2_{\max}\left((f_{\mathcal{J}_{k}}^{\prime}(x_k))^{\dagger}\right)$, $f(x_{\star})=0$, and $ \left\|  f_{\tau_h}(x_k)\right\|^{2}_{\infty}= \min \limits _{1\leq i\leq \nu}\left\|  f_{\tau_i}(x_k)\right\|^{2}_{\infty} $, then the iterations of the MR-BSNK2 method, i.e., the first case in \Cref{alg-The (MR/MD)-BSNK2 method}, satisfy
\begin{align}
\mathbb{E}\left[\left\| {x}_{k+1}- {x}_{\star}\right\|^{2}_2\right]
 \leq
\left(1-   \frac{\alpha\nu}{|\tau_h|}\frac{h_2^2(f^{\prime}_{\tau_h}\left(x_{k}\right))}{1+\eta^2}\right )
 \mathbb{E}\left[\left\|x_{k}-{x}_{\star}\right\|_2^{2}\right].\label{theorem:MR-BSNK2-1}
 \end{align}
\end{theorem}

\begin{proof}
Following a similar argument as in \Cref{theorem:MR-BSNK1} and from \Cref{alg-The (MR/MD)-BSNK2 method}, we get
\begin{align}
\left\|x_{k+1}-x_{\star}\right\|^{2}_2
&\leq
\left\|x_{k}-x_{\star}\right\|^{2}_2-  \alpha \sum\limits_{j\in\mathcal{J}_{k} }\left|  f_{j}(x_k)\right|^{2} \notag
\\
&=
\left\|x_{k}-x_{\star}\right\|^{2}_2- \alpha \sum\limits_{i=1}^{\nu}  \left\|  f_{\tau_i}(x_k)\right\|^{2}_{\infty}. \notag
\end{align}
Taking the expectation on both sides conditioned on $x_{k}$, we have
\begin{align}
\mathbb{E}^k\left[\left\|x_{k+1}-x_{\star}\right\|^{2}_2\right]
&\leq
\left\|x_{k}-x_{\star}\right\|^{2}_2- \alpha\mathbb{E}^k\left[ \sum\limits_{i=1}^{\nu}  \left\|  f_{\tau_i}(x_k)\right\|^{2}_{\infty}\right] \notag
\\
&=
\left\|x_{k}-x_{\star}\right\|^{2}_2-  \alpha\sum\limits_{i=1}^{\nu} \mathbb{E}^k\left[ \left\|  f_{\tau_i}(x_k)\right\|^{2}_{\infty}\right]. \notag
\end{align}
Noting that
\begin{align}
\sum\limits_{i=1}^{\nu} \mathbb{E}^k\left[ \left\|  f_{\tau_i}(x_k)\right\|^{2}_{\infty}\right]
&= \mathbb{E}^k[\|  f_{\tau_1}(x_k) \|_{\infty}^{2}]  +\mathbb{E}^k[ \|f_{\tau_2}(x_k) \|_{\infty}^{2}]+\cdots +\mathbb{E}^k [\|f_{\tau_\nu}(x_k) \|_{\infty}^{2}]\notag
\\
&=  \sum  \limits _{\tau_1 \in \binom{[m]}{\beta_1}} \frac{1}{\binom{m}{\beta_1}}   \| f_{\tau_1}(x_k)\|_{\infty}^{2}  +\sum  \limits _{\tau_2 \in \binom{[m]-\beta_1}{\beta_2}} \frac{1}{\binom{m-\beta_1}{\beta_2}}   \| f_{\tau_2}(x_k)\|_{\infty}^{2}              \notag
\\
&~~~~+\cdots +\sum  \limits _{\tau_{ \nu} \in \binom{[m]-\beta_1-\cdots-\beta_{\nu-1}}{\beta_\nu}} \frac{1}{\binom{m-\beta_1-\cdots-\beta_{\nu-1}}{\beta_\nu}}   \| f_{\tau_\nu}(x_k)\|_{\infty}^{2},  \notag
\end{align}
which together with the fact $\left\|  f_{\tau_h}(x_k)\right\|^{2}_{\infty}= \min \limits _{1\leq i\leq \nu}\left\|  f_{\tau_i}(x_k)\right\|^{2}_{\infty}$ gives
\begin{align}
\sum\limits_{i=1}^{\nu} \mathbb{E}^k\left[ \left\|  f_{\tau_i}(x_k)\right\|^{2}_{\infty}\right]
&\geq \sum  \limits _{\tau_1 \in \binom{[m]}{\beta_1}} \frac{1}{\binom{m}{\beta_1}}   \| f_{\tau_h}(x_k)\|_{\infty}^{2}  +\sum  \limits _{\tau_2 \in \binom{[m]-\beta_1}{\beta_2}} \frac{1}{\binom{m-\beta_1}{\beta_2}}   \| f_{\tau_h}(x_k)\|_{\infty}^{2}              \notag
\\
&~~~~+\cdots +\sum  \limits _{\tau_{ \nu} \in \binom{[m]-\beta_1-\cdots-\beta_{\nu-1}}{\beta_\nu}} \frac{1}{\binom{m-\beta_1-\cdots-\beta_{\nu-1}}{\beta_\nu}}   \| f_{\tau_h}(x_k)\|_{\infty}^{2}         \notag
\\
&=\nu \| f_{\tau_h}(x_k)\|_{\infty}^{2}    . \notag
\end{align}
Then, we have
\begin{align}
\mathbb{E}^k\left[\left\|x_{k+1}-x_{\star}\right\|^{2}_2\right]
&\leq
\left\|x_{k}-x_{\star}\right\|^{2}_2-  \alpha\nu \| f_{\tau_h}(x_k)\|_{\infty}^{2} \notag
\\
&\leq
\left\|x_{k}-x_{\star}\right\|^{2}_2-  \alpha\nu \frac{1}{|\tau_h|}\| f_{\tau_h}(x_k)\|_{2}^{2},\notag
\end{align}
which further together with (\ref{Sec2.7}) in \Cref{lemma3} gets
\begin{align}
\mathbb{E}^k\left[\left\|x_{k+1}-x_{\star}\right\|^{2}_2\right]
&\leq
\left\|x_{k}-x_{\star}\right\|^{2}_2-  \alpha\nu \frac{1}{|\tau_h|}\frac{1}{1+\eta^2}\left\|f^{\prime}_{\tau_h}\left(x_{k}\right)\left(x_{k}-x_{\star}\right)\right\|^2_2 \notag
\\
&\leq
\left\|x_{k}-x_{\star}\right\|^{2}_2-  \alpha\nu \frac{1}{|\tau_h|}\frac{1}{1+\eta^2}h_2^2(f^{\prime}_{\tau_h}\left(x_{k}\right))\left\| x_{k}-x_{\star} \right\|^2_2 \notag
\\
&=
\left(1-   \frac{\alpha\nu}{|\tau_h|}\frac{h_2^2(f^{\prime}_{\tau_h}\left(x_{k}\right))}{1+\eta^2}\right)\left\| x_{k}-x_{\star} \right\|^2_2 . \notag
\end{align}
Thus, taking expectation on both sides and using the tower rule of expectation, we get the desired result (\ref{theorem:MR-BSNK2-1}).
\end{proof}

\begin{theorem}
\label{theorem:MD-BSNK2}
If the nonlinear function $f $ satisfies the local tangential cone condition given in \Cref{definition}, $\eta=\max\limits _{i}
\eta_{i}<\frac{1}{2}$, $
\alpha=h_2^2((f_{\mathcal{J} }^{\prime}(x_k))^{\dagger})-2\eta\sigma^2_{\max}\left((f_{\mathcal{J} }^{\prime}(x_k))^{\dagger}\right)>0,$
where $h_2^2((f_{\mathcal{J}}^{\prime}(x_k))^{\dagger})= \min \limits _{\mathcal{J}_{k}} h_2^2((f_{\mathcal{J}_{k}}^{\prime}(x_k))^{\dagger})$ and $\sigma^2_{\max}\left((f_{\mathcal{J} }^{\prime}(x_k))^{\dagger}\right)= \max \limits _{\mathcal{J}_{k}}\sigma^2_{\max}\left((f_{\mathcal{J}_{k}}^{\prime}(x_k))^{\dagger}\right)$, $f(x_{\star})=0$, and $ \left\|  u_{\tau_h}(x_k)\right\|^{2}_{\infty}= \min \limits _{1\leq i\leq \nu}\left\|  u_{\tau_i}(x_k)\right\|^{2}_{\infty} $, then the iterations of the MD-BSNK2 method, i.e., the second case in \Cref{alg-The (MR/MD)-BSNK2 method}, satisfy
\begin{align}
\mathbb{E}\left[\left\| {x}_{k+1}- {x}_{\star}\right\|^{2}_2\right]
 \leq
\left(1-   \min \limits _{i\in[m]} \left\|\nabla f_{i}\left(x_{k}\right)\right\|_2^{2}\frac{\alpha\nu}{|\tau_h|}\frac{h_2^2(G_{\tau_h}\left(x_{k}\right))}{1+\eta^2}\right )
 \mathbb{E}\left[\left\|x_{k}-{x}_{\star}\right\|_2^{2}\right].\label{theorem:MD-BSNK2-1}
 \end{align}
\end{theorem}

\begin{proof}
Following the proof of \Cref{theorem:MR-BSNK1} and according to \Cref{alg-The (MR/MD)-BSNK2 method}, we have
\begin{align}
\left\|x_{k+1}-x_{\star}\right\|^{2}_2
&\leq
\left\|x_{k}-x_{\star}\right\|^{2}_2-  \alpha \sum\limits_{j\in\mathcal{J}_{k} }\left|  f_{j}(x_k)\right|^{2} \notag
\\
&=
\left\|x_{k}-x_{\star}\right\|^{2}_2-  \alpha \sum\limits_{j\in\mathcal{J}_{k} }\frac{\left|  f_{j}(x_k)\right|^{2}}{\left\|\nabla f_{j}\left(x_{k}\right)\right\|_2^{2}}\left\|\nabla f_{j}\left(x_{k}\right)\right\|_2^{2} \notag
\\
&\leq
\left\|x_{k}-x_{\star}\right\|^{2}_2-  \alpha \min \limits _{i\in[m]} \left\|\nabla f_{i}\left(x_{k}\right)\right\|_2^{2} \sum\limits_{j\in\mathcal{J}_{k} }\frac{\left|  f_{j}(x_k)\right|^{2}}{\left\|\nabla f_{j}\left(x_{k}\right)\right\|_2^{2}} \notag
\\
&\leq
\left\|x_{k}-x_{\star}\right\|^{2}_2-\alpha  \min \limits _{i\in[m]} \left\|\nabla f_{i}\left(x_{k}\right)\right\|_2^{2} \sum\limits_{j\in\mathcal{J}_{k} } \left|u_{j}(x_k)\right|^{2}  \notag
\\
&=
\left\|x_{k}-x_{\star}\right\|^{2}_2- \alpha\min \limits _{i\in[m]} \left\|\nabla f_{i}\left(x_{k}\right)\right\|_2^{2}  \sum\limits_{i=1}^{\nu}  \left\|  u_{\tau_i}(x_k)\right\|^{2}_{\infty}. \notag
\end{align}
Now, taking the expectation on both sides conditioned on $x_{k}$, we have
\begin{align}
\mathbb{E}^k\left[\left\|x_{k+1}-x_{\star}\right\|^{2}_2\right]
&\leq
\left\|x_{k}-x_{\star}\right\|^{2}_2- \alpha\min \limits _{i\in[m]} \left\|\nabla f_{i}\left(x_{k}\right)\right\|_2^{2} \mathbb{E}^k\left[ \sum\limits_{i=1}^{\nu}  \left\|  u_{\tau_i}(x_k)\right\|^{2}_{\infty}\right] \notag
\\
&=
\left\|x_{k}-x_{\star}\right\|^{2}_2-  \alpha\min \limits _{i\in[m]} \left\|\nabla f_{i}\left(x_{k}\right)\right\|_2^{2} \sum\limits_{i=1}^{\nu} \mathbb{E}^k\left[ \left\|  u_{\tau_i}(x_k)\right\|^{2}_{\infty}\right]. \notag
\end{align}
Noting that
\begin{align}
\sum\limits_{i=1}^{\nu} \mathbb{E}^k\left[ \left\|  u_{\tau_i}(x_k)\right\|^{2}_{\infty}\right]
&= \mathbb{E}^k[\|  u_{\tau_1}(x_k) \|_{\infty}^{2}]  +\mathbb{E}^k[ \|u_{\tau_2}(x_k) \|_{\infty}^{2}]+\cdots +\mathbb{E}^k [\|u_{\tau_\nu}(x_k) \|_{\infty}^{2}]\notag
\\
&=  \sum  \limits _{\tau_1 \in \binom{[m]}{\beta_1}} \frac{1}{\binom{m}{\beta_1}}   \| u_{\tau_1}(x_k)\|_{\infty}^{2}  +\sum  \limits _{\tau_2 \in \binom{[m]-\beta_1}{\beta_2}} \frac{1}{\binom{m-\beta_1}{\beta_2}}   \| u_{\tau_2}(x_k)\|_{\infty}^{2}              \notag
\\
&~~~~+\cdots +\sum  \limits _{\tau_{ \nu} \in \binom{[m]-\beta_1-\cdots-\beta_{\nu-1}}{\beta_\nu}} \frac{1}{\binom{m-\beta_1-\cdots-\beta_{\nu-1}}{\beta_\nu}}   \| u_{\tau_\nu}(x_k)\|_{\infty}^{2},  \notag
\end{align}
which together with the fact $\left\|  u_{\tau_h}(x_k)\right\|^{2}_{\infty}= \min \limits _{1\leq i\leq \nu}\left\|  u_{\tau_i}(x_k)\right\|^{2}_{\infty}$ gives
\begin{align}
\sum\limits_{i=1}^{\nu} \mathbb{E}^k\left[ \left\|  u_{\tau_i}(x_k)\right\|^{2}_{\infty}\right]
&\geq \sum  \limits _{\tau_1 \in \binom{[m]}{\beta_1}} \frac{1}{\binom{m}{\beta_1}}   \| u_{\tau_h}(x_k)\|_{\infty}^{2}  +\sum  \limits _{\tau_2 \in \binom{[m]-\beta_1}{\beta_2}} \frac{1}{\binom{m-\beta_1}{\beta_2}}   \| u_{\tau_h}(x_k)\|_{\infty}^{2}              \notag
\\
&~~~~+\cdots +\sum  \limits _{\tau_{ \nu} \in \binom{[m]-\beta_1-\cdots-\beta_{\nu-1}}{\beta_\nu}} \frac{1}{\binom{m-\beta_1-\cdots-\beta_{\nu-1}}{\beta_\nu}}   \| u_{\tau_h}(x_k)\|_{\infty}^{2}         \notag
\\
&=\nu \| u_{\tau_h}(x_k)\|_{\infty}^{2}    . \notag
\end{align}
Then, we have
\begin{align}
\mathbb{E}^k\left[\left\|x_{k+1}-x_{\star}\right\|^{2}_2\right]
&\leq
\left\|x_{k}-x_{\star}\right\|^{2}_2-  \alpha\min \limits _{i\in[m]} \left\|\nabla f_{i}\left(x_{k}\right)\right\|_2^{2} \nu \| u_{\tau_h}(x_k)\|_{\infty}^{2} \notag
\\
&\leq
\left\|x_{k}-x_{\star}\right\|^{2}_2-  \alpha\min \limits _{i\in[m]} \left\|\nabla f_{i}\left(x_{k}\right)\right\|_2^{2} \nu \frac{1}{|\tau_h|}\| u_{\tau_h}(x_k)\|_{2}^{2},\notag
\end{align}
which further together with (\ref{Sec2.7.1}) in \Cref{lemma4} gives
\begin{align}
\mathbb{E}^k\left[\left\|x_{k+1}-x_{\star}\right\|^{2}_2\right]
&\leq
\left\|x_{k}-x_{\star}\right\|^{2}_2-  \alpha\min \limits _{i\in[m]} \left\|\nabla f_{i}\left(x_{k}\right)\right\|_2^{2} \nu \frac{1}{|\tau_h|}\frac{1}{1+\eta^2}\left\|G_{\tau_h}\left(x_k\right)\left(x_k-x_{\star}\right)\right\|^2_2 \notag
\\
&\leq
\left\|x_{k}-x_{\star}\right\|^{2}_2-  \alpha\min \limits _{i\in[m]} \left\|\nabla f_{i}\left(x_{k}\right)\right\|_2^{2} \nu \frac{1}{|\tau_h|}\frac{1}{1+\eta^2}h_2^2(G_{\tau_h}\left(x_{k}\right))\left\| x_{k}-x_{\star} \right\|^2_2 \notag
\\
&=
\left(1-   \min \limits _{i\in[m]} \left\|\nabla f_{i}\left(x_{k}\right)\right\|_2^{2}\frac{\alpha\nu}{|\tau_h|}\frac{h_2^2(G_{\tau_h}\left(x_{k}\right))}{1+\eta^2}\right)\left\| x_{k}-x_{\star} \right\|^2_2 . \notag
\end{align}
Thus, taking expectation on both sides and using the tower rule of expectation, we get the desired result (\ref{theorem:MD-BSNK2-1}).
\end{proof}

\begin{remark}
\label{remark_th_MD-BSNK2:1}
According to \Cref{theorem:MR-BSNK2,theorem:MD-BSNK2}, we know that the convergence factors of the MR-BSNK2 and MD-BSNK2 methods are
$$
\rho_{\text{MR-BSNK2}}= 1-   \frac{\alpha\nu}{|\tau_h|}\frac{ h_2^2(f^{\prime}_{\tau_h}\left(x_{k}\right))}{1+\eta^2},
\quad
\text{and}
\quad
\rho_{\text{MD-BSNK2}}= 1-   \frac{\alpha\nu}{|\tau_h|}\frac{\min \limits _{i\in[m]} \left\|\nabla f_{i}\left(x_{k}\right)\right\|_2^{2} h_2^2(G_{\tau_h}\left(x_{k}\right))}{1+\eta^2} ,
$$
respectively.
Then, similar to the deduction in \Cref{remark_th_MD-BSNK1-1}, we can conclude that the convergence factor of the MR-BSNK2 method is smaller than that of the MD-BSNK2 method.
\end{remark}

\section{Experimental results}
\label{sec:experiments}
In this section, we investigate the convergence behavior of the proposed methods, i.e., the MR-SNK, MD-SNK, MR-BSNK1, MD-BSNK1, MR-BSNK2 and MD-BSNK2 methods, in solving the problems including brown almost linear function and generalized linear model (GLM).  We mainly compare our six methods with the existing methods in terms of the iteration numbers (denoted as ``IT'') and computing time in seconds (denoted as ``CPU''). Here, it should be pointed out that IT and CPU are the averages of IT and CPU of ten runs of algorithm. All experiments terminate once the residual at $x_{k}$, defined by $ \left\|f(x_{k})\right\|_2^2 $ is less than $10^{-6}$, or the number of iterations exceeds 200000.

\subsection{Brown almost linear function}
The function has the following form
\label{sec:Brown almost linear function}
$$
\begin{array}{lr}
f_k(x)=x^{(k)}+\sum_{i=1}^n x^{(i)}-(n+1), & 1 \leq k<n;\\
f_k(x)=\left(\prod_{i=1}^n x^{(i)}\right)-1, & k=n,
\end{array}
$$
which can be found in \cite{more1981testing, wang2022nonlinear}. Since the authors in \cite{wang2022nonlinear} concluded that the convergence behavior of the NRK method outperforms the NURK and SGD methods, we only need to compare our methods with the NRK method and all experiments are initial from $x_{0}=0.5*ones(n,1)$.
\subsubsection{The impact of parameters $\beta$ and $\nu$ on our methods}
\label{sec:The impact of parameters}
From \Cref{alg2,alg-The (MR/MD)-BSNK1 method,alg-The (MR/MD)-BSNK2 method}, we know that two parameters $\beta$ and $\nu$ affect the numerical results of our methods. With this in mind, we first show in \Cref{fig_variation_beta} how the parameter $\beta$ impacts the MR-SNK, MD-SNK, MR-BSNK1 and MD-BSNK1 methods and show in \Cref{fig_variation_nu} how sensitive the MR-BSNK2 and MD-BSNK2 methods are for variation of parameter $\nu$.
\begin{figure}[ht]
 \begin{center}
\includegraphics [height=3.8cm,width=9cm ]{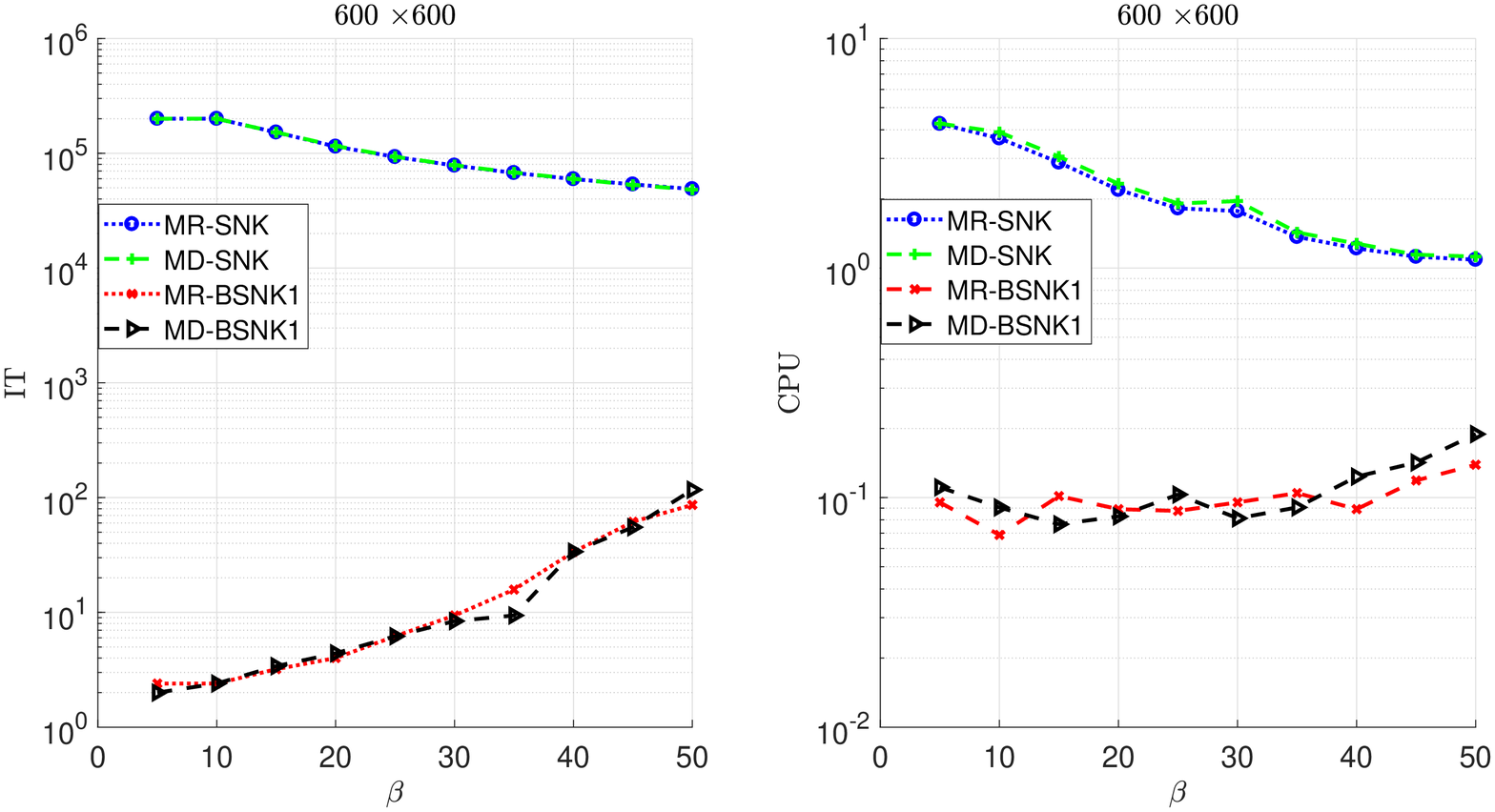}
 \end{center}
\caption{IT (left) and CPU (right) versus $\beta$  for the MR-SNK, MD-SNK, MR-BSNK1 and MD-BSNK1 methods.}\label{fig_variation_beta}
\end{figure}
\begin{figure}[ht]
 \begin{center}
\includegraphics [height=3.8cm,width=9cm ]{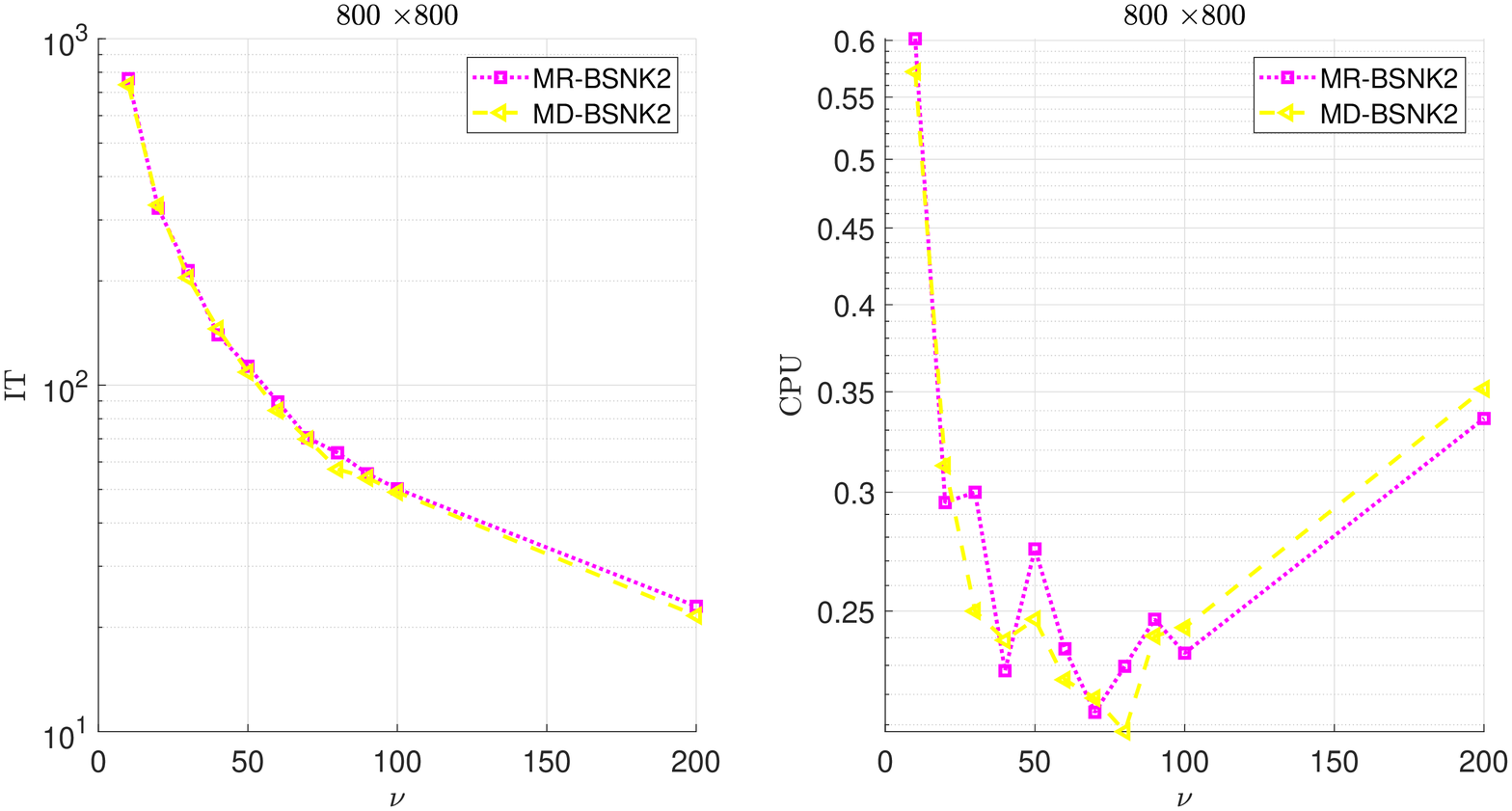}
 \end{center}
\caption{IT (left) and CPU (right) versus $\nu$ for the MR-BSNK2 and MD-BSNK2 methods.}\label{fig_variation_nu}
\end{figure}

In \Cref{fig_variation_beta}, both IT and CPU of the MR-SNK and MD-SNK methods drop as $\beta$ increases from 5 to 50. Meanwhile, we find that the iteration numbers of the MR-BSNK1 and MD-BSNK1 methods increase as $\beta$ increases, while their CPU times first reach their lowest point at $\beta=10$ and $\beta=15$, respectively, and then keep increasing as $\beta$ increases with some mild rebound. In addition, 
two multiple samples-based block methods outperform the two single sample-based methods in terms of iteration numbers and CPU time as expected.

Intuitively, as the sample size used in each iteration increases, the number of iteration steps is expected to decrease, which fits well with \Cref{fig_variation_nu} and is also our motivation for introducing block updating. However, the CPU times of the MR-BSNK2 and MD-BSNK2 methods exhibit a U-shaped curve as $\nu$ increases. That is, 
it first declines and then rises, albeit with some mild rebound. Actually, if the block size $\nu$ is small, then these two methods need more iteration numbers to achieve the desired accuracy, which in turn requires more CPU times. On the contrary, if the parameter $\nu$ is large, then more CPU time is needed to spend on each updating step. Hence although the number of iterations is relatively small, the total time is still high.

In a word, choosing a suitable parameter $\beta$ or $\nu$ is crucial for our methods. Thus, how to select a optimal parameter is a very interesting research topic, but the corresponding theoretical analysis is currently unavailable, which can be regarded as a future work.

\subsubsection{Comparison of our methods and the NRK method}
\label{sec: Comparison of our methods and the NRK method}
In this subsection, we demonstrate the effectiveness of our proposed methods by comparing them with the NRK method. The iteration numbers and CPU time for these methods are presented in \Cref{tab_Brown_IT,tab_Brown_CPU}, respectively. As shown in \Cref{tab_Brown_IT,tab_Brown_CPU}, although the optimal parameters are not always selected during the experiments, our methods, i.e., the MR-SNK, MD-SNK, MR-BSNK1, MD-BSNK1, MR-BSNK2 and MD-BSNK2 methods, usually require less iteration numbers and computing time than that of the NRK method. Moreover, the block methods are always outperform the single sample-based methods in terms of iteration numbers and CPU time. Most interesting, the block methods even achieve a time speedup of 200 times compared to the NRK method, as shown in the last row of \Cref{tab_Brown_CPU}.
\begin{table}[tp]
  \centering
  \fontsize{6.7}{6.7} \selectfont
    \caption{IT comparison of our methods and the NRK method.}
    \label{tab_Brown_IT}
    \begin{tabular}{ccccccccccc}
    \hline
    $m \times n$    &$\beta$ & $ \nu $&NRK& MR-SNK&  MD-SNK& MR-BSNK1& MD-BSNK1&  MR-BSNK2 & MD-BSNK2\cr
 \hline
$50\times 50$   & 5 & 5 &4738.5 & 3545.1 & 3430.3 & 117.4 & 172.4 & 111.6 &113.4   \cr
$80\times 80$   &5 & 5 &1073.2 & 8327.1 & 8565.5 &     93&30.1 & 171.4 & 188.7 \cr
$150\times 150$ & 10 & 10 & 33810 &15093&15074& 4.1& 6.2 &140.6& 147 \cr
$200\times 200$ &10 & 10&56954&26135&26309&3.3&3.6&177.8& 187.6 \cr
$400\times 400$ &20& 20 & 197486 & 52110& 52278     &     9.4& 3   & 157  &146.8 \cr\hline
\end{tabular}
\end{table}

\begin{table}[tp]
  \centering
  \fontsize{6.7}{6.7} \selectfont
    \caption{CPU comparison of our methods and the NRK method.}
    \label{tab_Brown_CPU}
    \begin{tabular}{ccccccccccc}
    \hline
    $m \times n$    &$\beta$ & $ \nu $   &NRK& MR-SNK&  MD-SNK& MR-BSNK1& MD-BSNK1&  MR-BSNK2 & MD-BSNK2\cr
 \hline
$50\times 50$   &5 & 5 & 0.3359&0.0688&0.0594&0.0500&0.0547&0.0422&0.0297 \cr
$80\times 80$   &   5 & 5  & 0.5000&0.1437&0.1469&0.0453&0.0203&0.0531&0.0266  \cr
$150\times 150$ &10 & 10   & 1.3047&0.2422& 0.2578&0.0297&0.0219&0.0422&0.0469 \cr
$200\times 200$ &10 & 10     & 2.5328&0.5656&0.4953&0.0047& 0.0125&0.0578&0.0688  \cr
$400\times 400$ &  20 & 20  &    8.0906&0.8313&0.9500&0.0484 &0.0437&0.0750&0.0703 \cr\hline
\end{tabular}
\end{table}

\subsection{GLM}
\label{sec:logistic regression}
Here we compare our methods against the sketched Newton-Raphson (SNR) \cite{yuan2022sketched} for solving regularized GLM, i.e., $w^*=\arg\min \limits _{w \in \mathbb{R}^d} P(w) \stackrel{\text { def }}{=} \frac{1}{p} \sum_{i=1}^p \phi_i\left(a_i^{\top} w\right)+\frac{\lambda}{2}\|w\|^2$, where $\phi_i(t)=\ln \left(1+\mathrm{e}^{-y_i t}\right)$ is the logistic loss, $y_i\in\{-1, 1\}$ is the $i$th target value, $a_i \in \mathbb{R}^d$ are data samples for $i=1, \cdots, p$, and $w \in \mathbb{R}^d$ is the parameter to optimize. Using the processing procedure in \cite{yuan2022sketched}, the original GLM can be equivalently transformed into 
solving the following nonlinear problem

$$
f(x) \stackrel{\text { def }}{=}\left[\begin{array}{c}
\frac{1}{\lambda p} A \alpha-w \\
\alpha+\Phi(w)
\end{array}\right]=0
,$$
where $f: \mathbb{R}^{p+d} \rightarrow \mathbb{R}^{p+d}$, $x=\left[\begin{array}{c}
 \alpha \\
w
\end{array}\right]\in \mathbb{R}^{p+d} $, $A \stackrel{\text { def }}{=}\left[\begin{array}{lll}a_1 & \cdots & a_p\end{array}\right] \in \mathbb{R}^{d \times p}$, $\Phi(w) \stackrel{\text { def }}{=}\left[\phi_1^{\prime}\left(a_1^{T} w\right) \cdots \phi_p^{\prime}\left(a_p^{T} w\right)\right]^{T} \in \mathbb{R}^p$ and $\alpha=-\Phi(w)$. The advantage of this transformation is to ensure the sampling of a single sample and avoid a full passes over the data.

All datasets applied for GLM are 
obtained from \cite{chang2011libsvm} on \url{https://www.csie.ntu.edu.tw/~cjlin/libsvmtools/datasets/} and the scaled versions are used if provided. These datasets are either ill or well conditioned, dense or sparse and their details including smoothness constant $L$ of the model, density and condition number of the data matrix $A\in \mathbb{R}^{d\times p}$ are provided in \Cref{tab_properties of dataset}. Here smoothness constant $L$ is defined by $L\stackrel{\text { def }}{=}\frac{ \lambda_{\max}(AA^T)}{4p}+\lambda$ and density is computed by $$\text{density}\stackrel{\text { def }}{=}\frac{\text{number of nonzero of a $d\times p$ data matrix $A$}}{dp}.$$ For all methods, we use $\lambda=\frac{1}{p}$ as the regularization parameter and let the initial value $x_0$ to be zero, i.e., $w_0 =0\in \mathbb{R}^d$ and $\alpha_0 =0\in \mathbb{R}^p$.
\begin{table}[]
\centering
   \fontsize{8}{8}\selectfont
       \caption{Details of the data sets for GLM.}
    \label{tab_properties of dataset}
    \begin{tabular}{ c| c c c c  c }
 \hline
   dataset       &dimension $(d)$   &  samples $(p)$    & $L$    & condition number  &density\cr \hline
   fourclass     &     2            &  862              & 0.0824        &  1.0757           &0.9959\cr
   german.numer  &    24            & 1000              &  2.1113       & 15.4082           &0.9584\cr
   heart         &    13            &  270              & 0.6973        &  7.0996           &0.9624\cr
   ionosphere    &    34            &  351              &1.5290         & 2.4485e+17        &0.8841\cr
   splice        &    60            & 1000              & 0.4349        &  2.6377           &1\cr
   sonar         &    60            & 208               & 3.2282        &  89.9388          &0.9999\cr
   w3a           &    300           & 4912              & 0.6436        &3.9777e+33         &0.0388\cr
   w4a           &    300           & 7366              & 0.6385        & 1.0116e+34        &0.0389\cr
   \hline
\end{tabular}
\end{table}

To make a more reasonable and fair comparison with the SNR method, we will discuss the convergence behavior of single-sample iteration and multiple-sample iteration respectively. Specifically, the iteration numbers and computing time of the MR-SNK, MD-SNK and Kaczmarz-TCS methods are compared in \Cref{fig_single_sample_IT,fig_single_sample_CPU}, respectively. Here the Kaczmarz-TCS method is a special variation of the SNR method and its detailed implementation can be found in Algorithm 3 in \cite{yuan2020sketched}. In \Cref{fig_multi_samples}, we compare the numerical results of block methods, i.e., the MR-BSNK1, MD-BSNK1, MR-BSNK2, MD-BSNK2 and Block TCS methods, where the last method is a block version of the Kaczmarz-TCS method; refer to Algorithm 4 in \cite{yuan2020sketched} for more details. In the specific experiments, for the Block TCS method, we set $\gamma=1$, $\tau_d=d$, $\tau_p=150$ and the Bernoulli parameter $b=\frac{p}{(p+\tau_p)}-0.11$. Here, it is worth noting that, in order to demonstrate the advantage of greedy sampling over uniform sampling and the better use of the structure of the function $f$, our block methods adopts the same iterative framework as the Block TCS method. That is, for the first $d$ rows of $f$, we compute its least norm solution directly, which is the same as the Block TCS method because $\tau_d=d$; while for the last $p$ rows of $f$, we adopt greedy updating strategies discussed in \Cref{alg-The (MR/MD)-BSNK1 method,alg-The (MR/MD)-BSNK2 method} but the Block TCS method uses uniform sampling.
\begin{figure}[ht]
 \begin{center}
\includegraphics [height=3.8cm,width=4.5cm ]{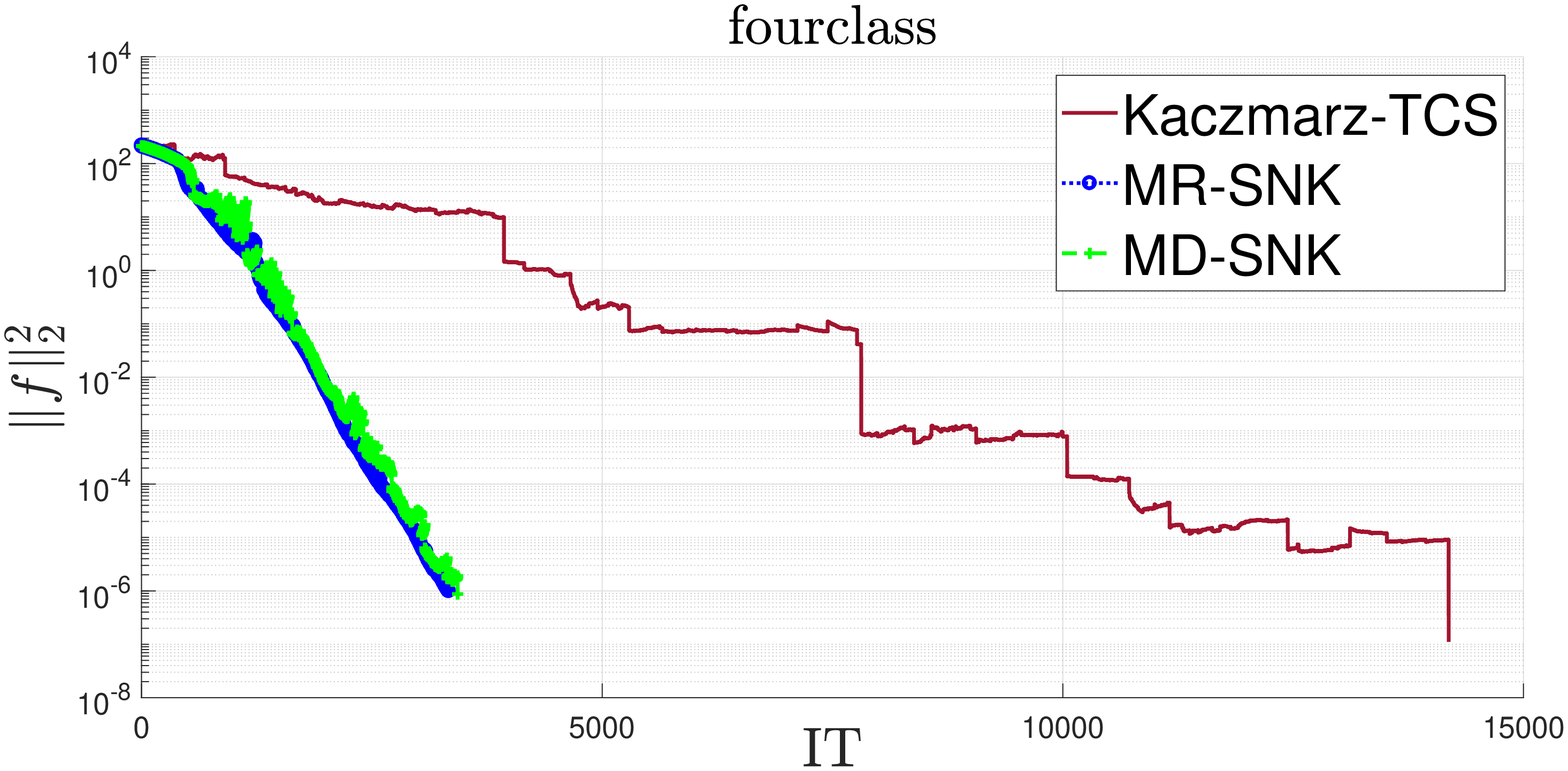}
\includegraphics [height=3.8cm,width=4.5cm ]{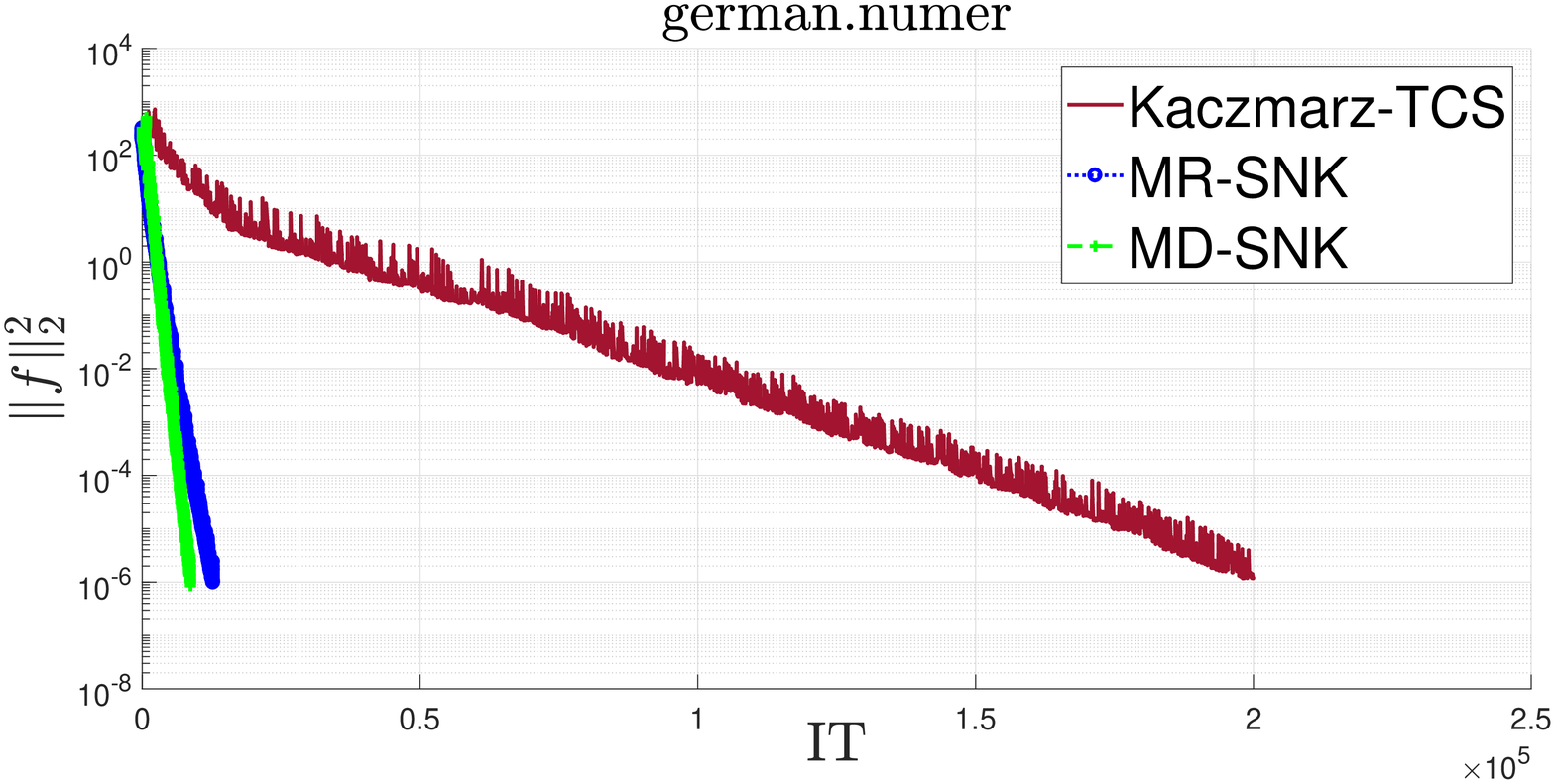}
\includegraphics [height=3.8cm,width=4.5cm ]{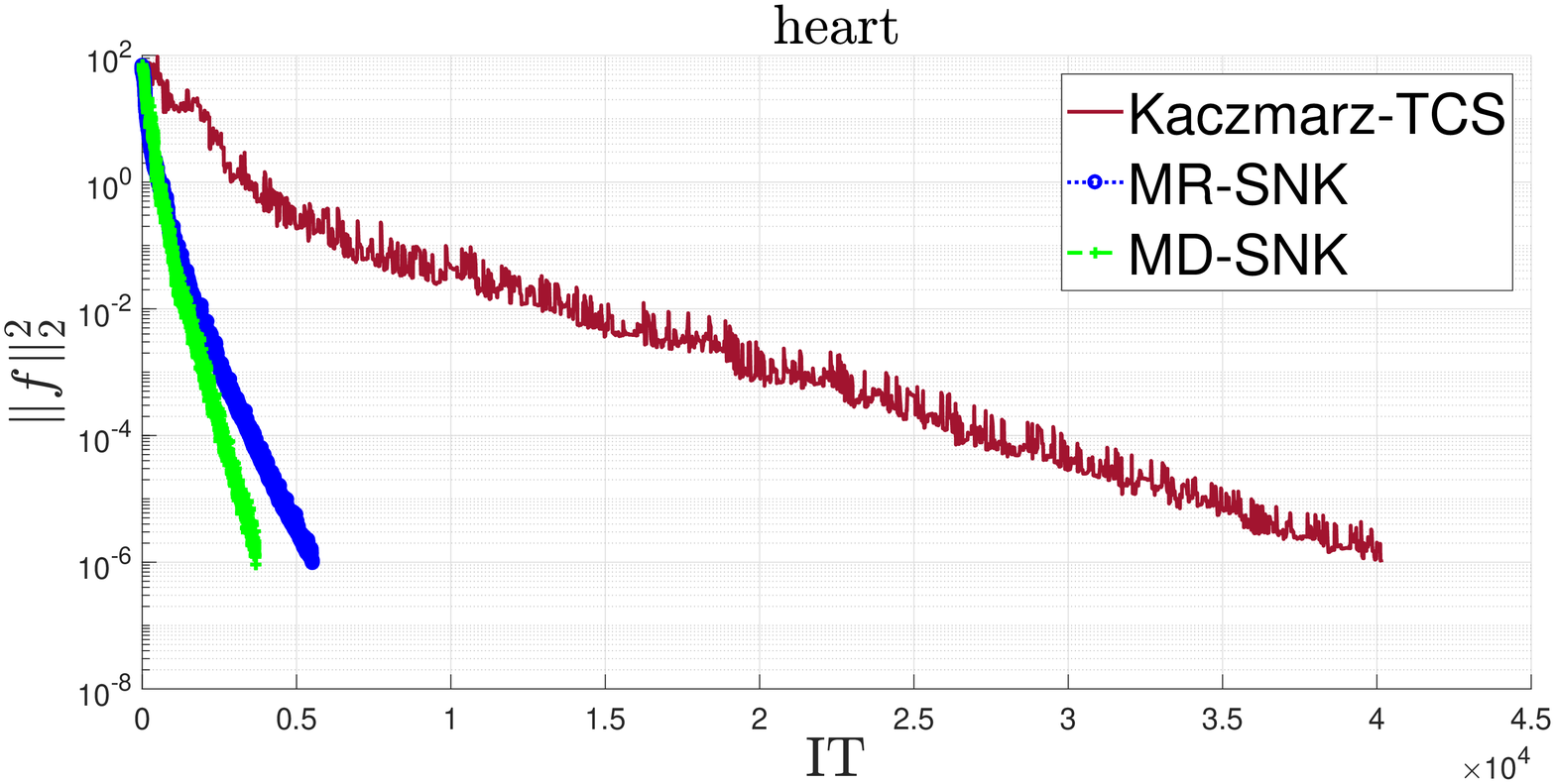}
\includegraphics [height=3.8cm,width=4.5cm ]{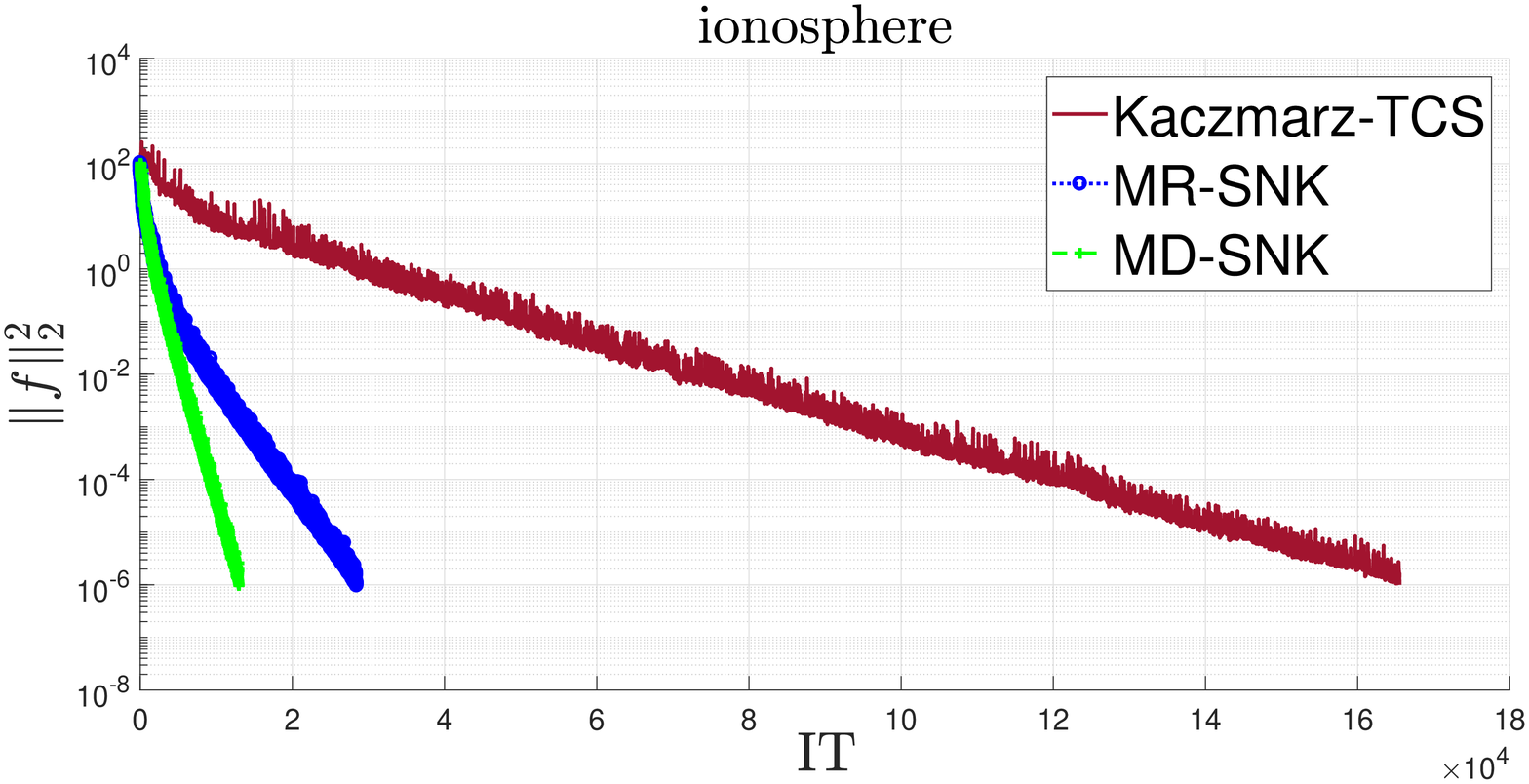}
\includegraphics [height=3.8cm,width=4.5cm ]{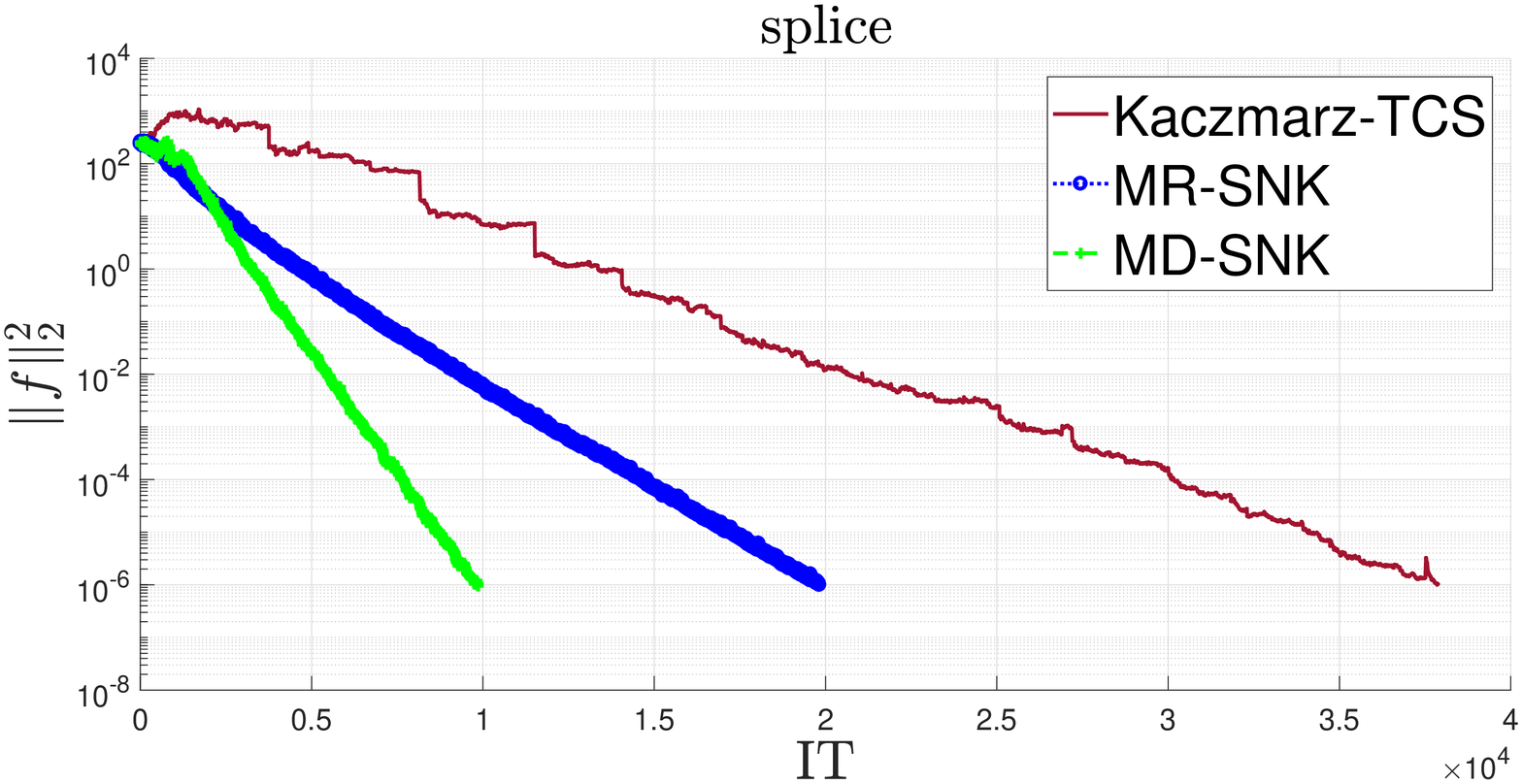}
\includegraphics [height=3.8cm,width=4.5cm ]{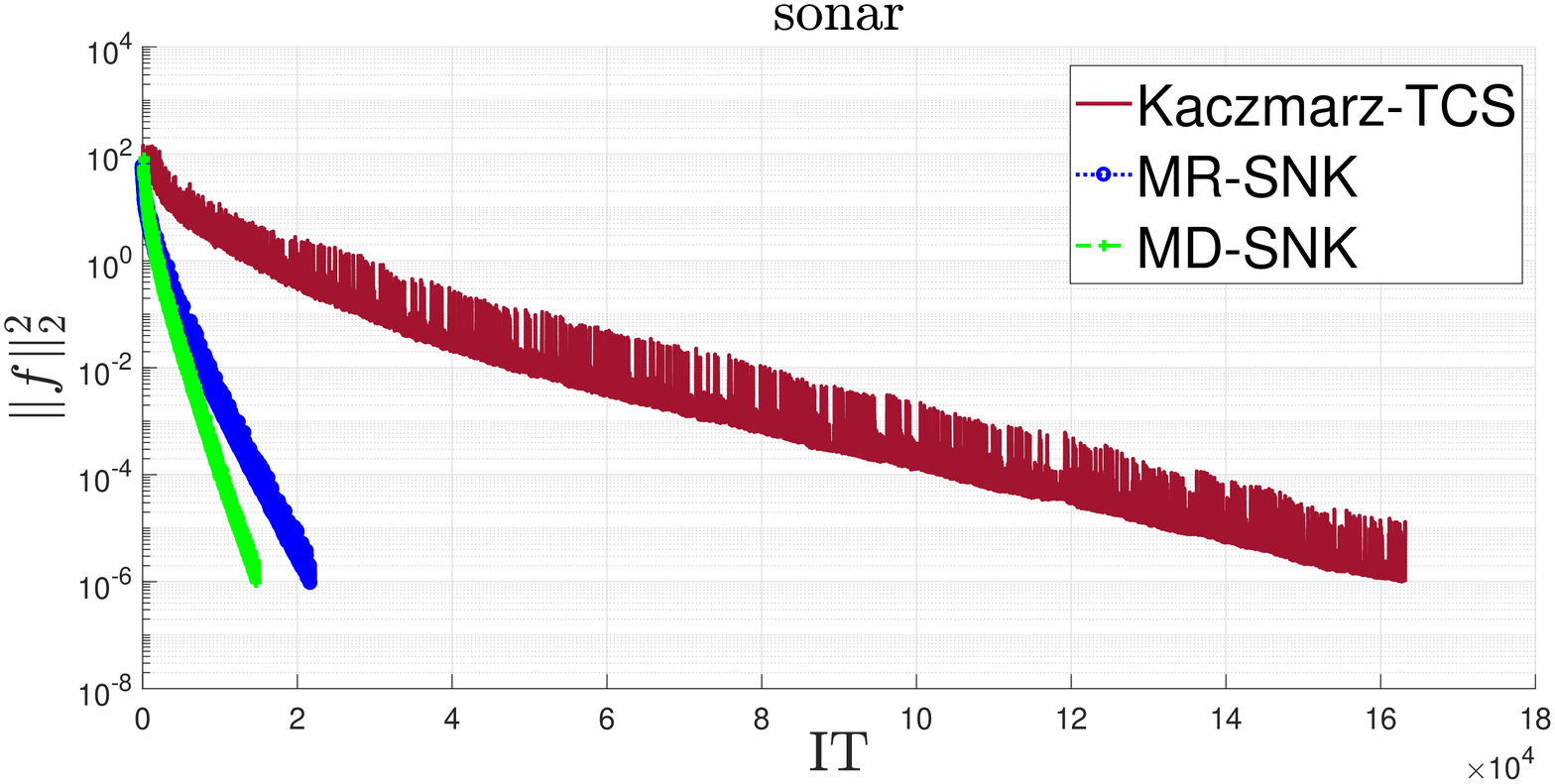}
 \end{center}
\caption{$\|f\|_2^2$ versus IT for the Kaczmarz-TCS, MR-SNK and MD-SNK methods with $\beta=80$.}\label{fig_single_sample_IT}
\end{figure}
\begin{figure}[ht]
 \begin{center}
\includegraphics [height=3.8cm,width=4.5cm ]{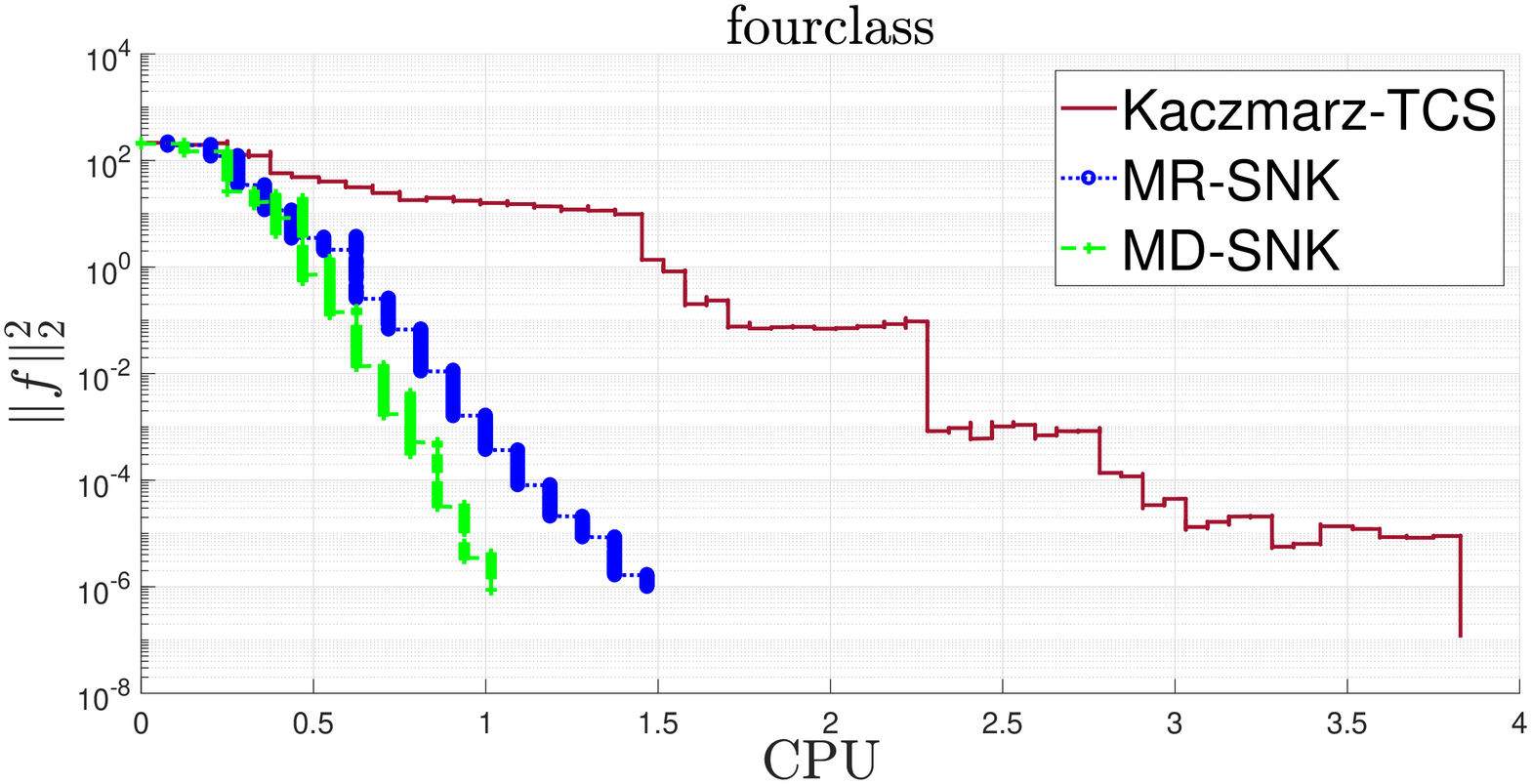}
\includegraphics [height=3.8cm,width=4.5cm ]{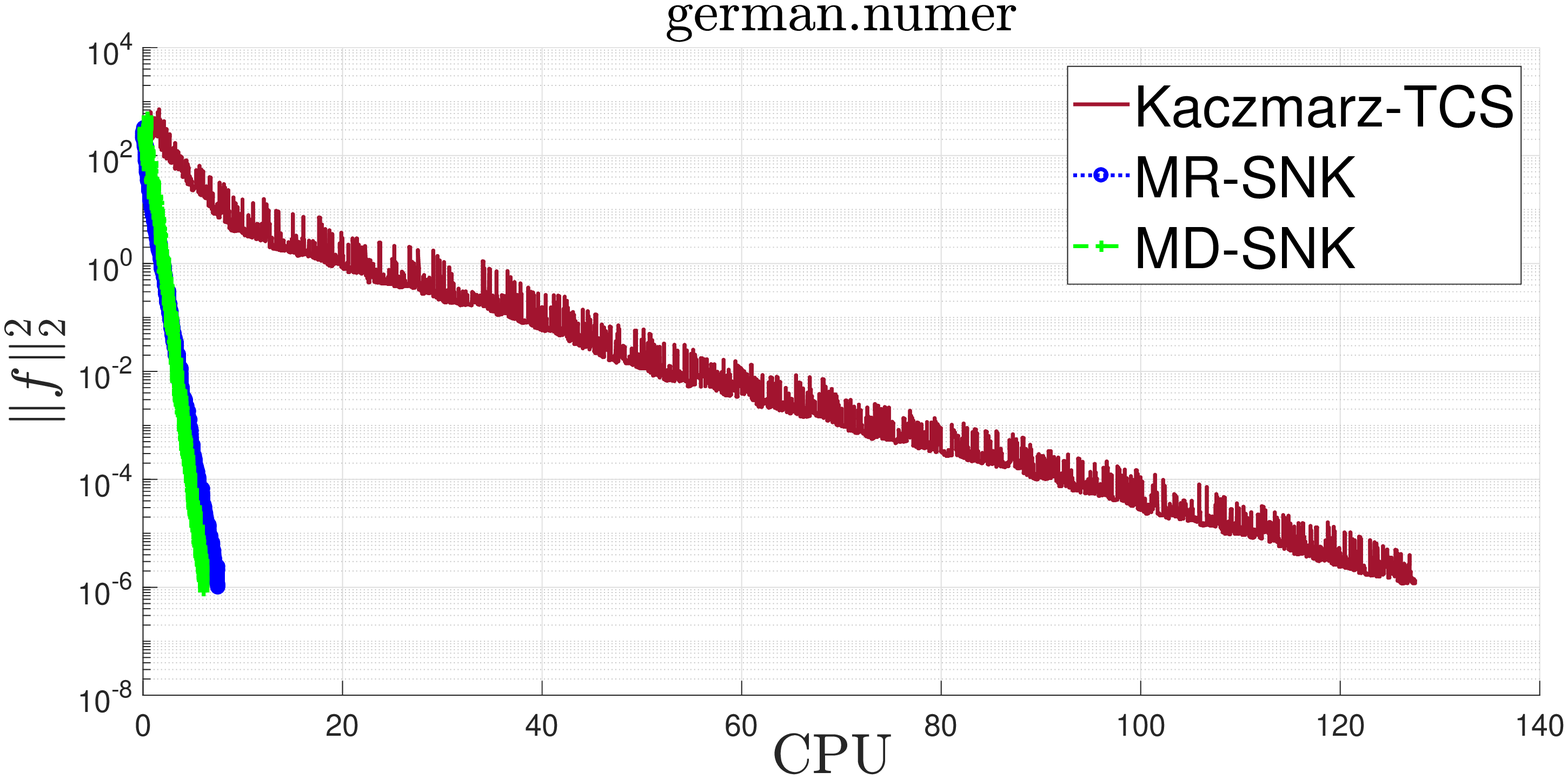}
\includegraphics [height=3.8cm,width=4.5cm ]{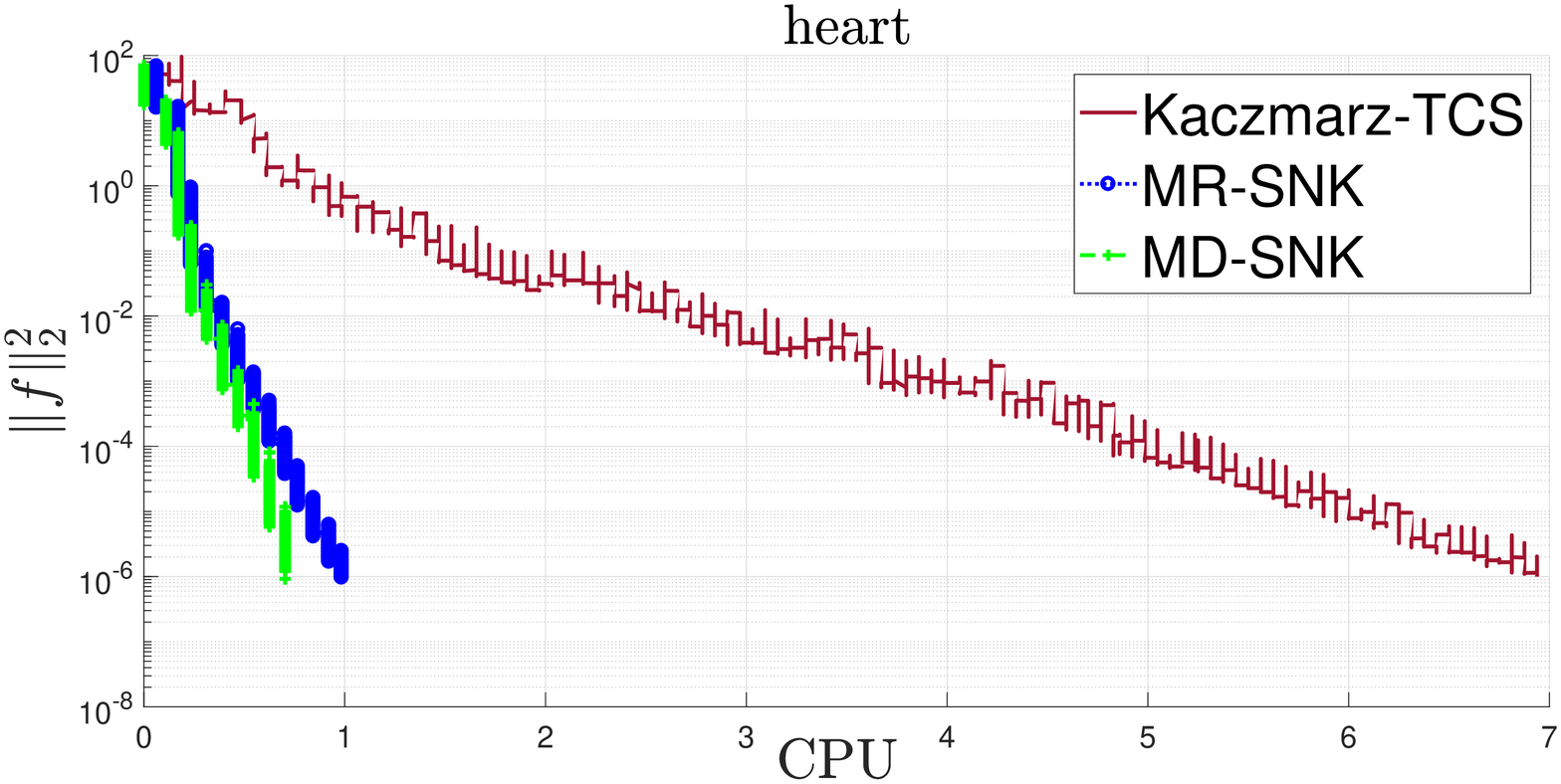}
\includegraphics [height=3.8cm,width=4.5cm ]{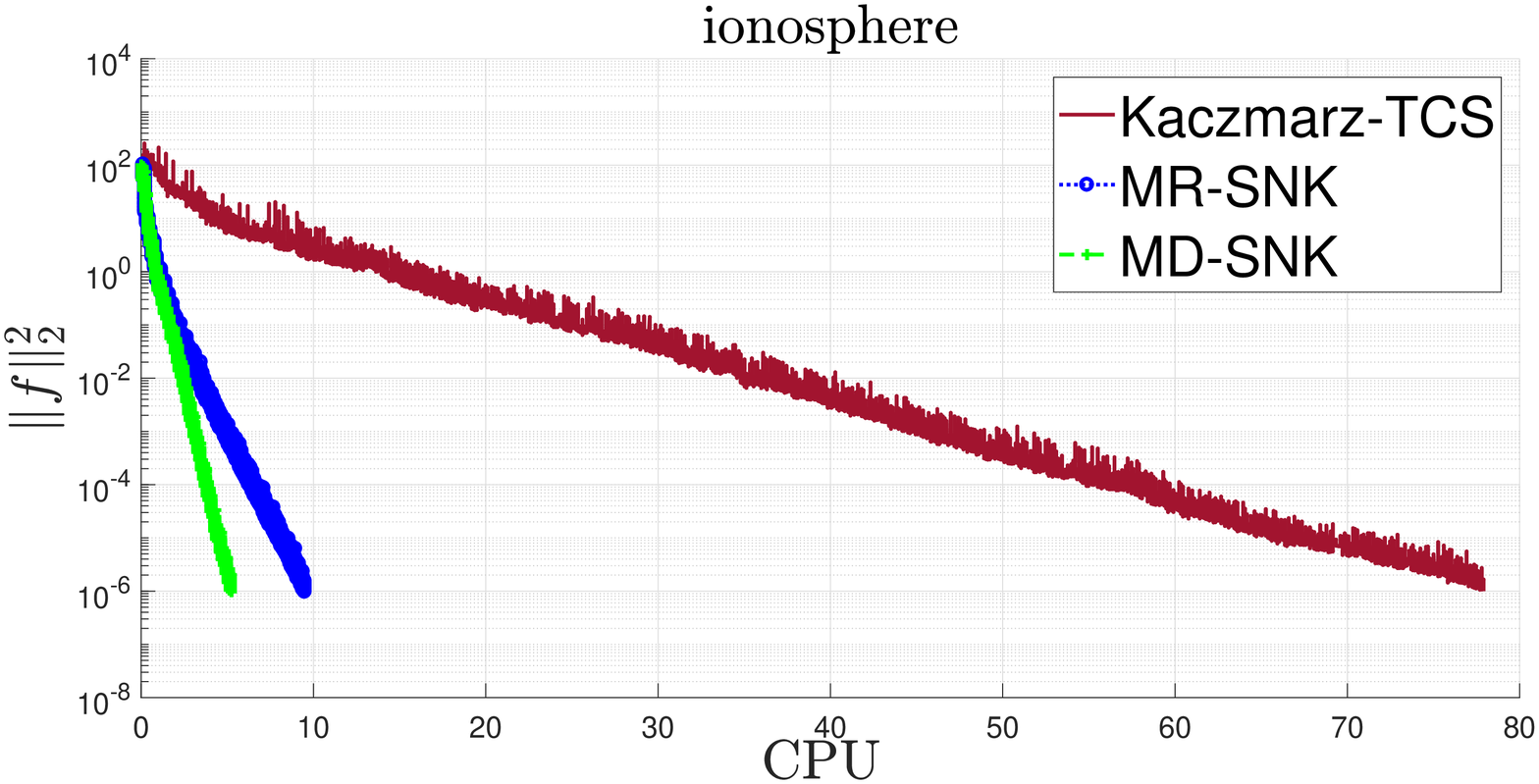}
\includegraphics [height=3.8cm,width=4.5cm ]{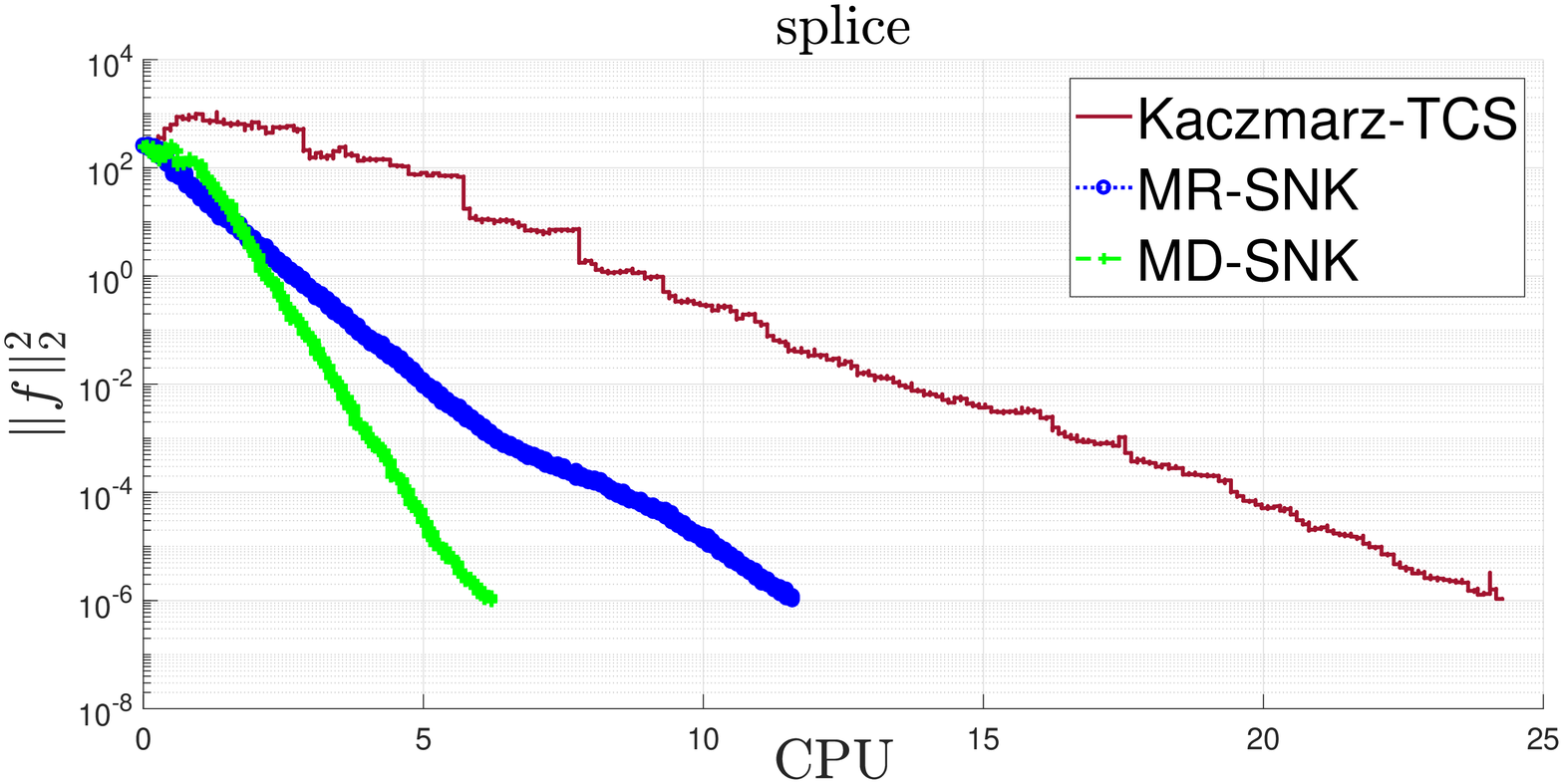}
\includegraphics [height=3.8cm,width=4.5cm ]{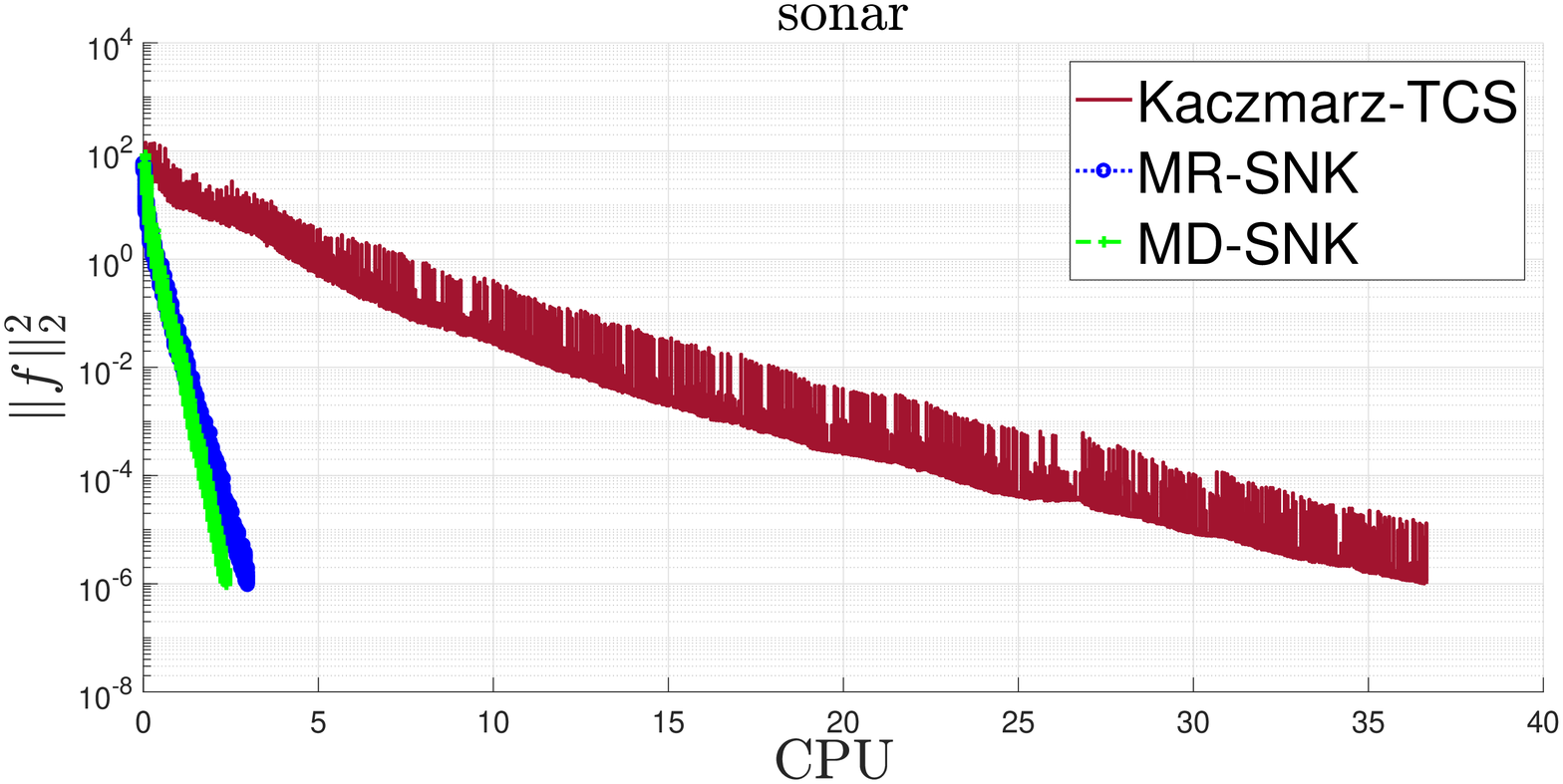}
 \end{center}
\caption{$\|f\|_2^2$ versus CPU for the Kaczmarz-TCS, MR-SNK and MD-SNK methods with $\beta=80$.}\label{fig_single_sample_CPU}
\end{figure}
\begin{figure}[ht]
 \begin{center}
\includegraphics [height=3.5cm,width=7cm ]{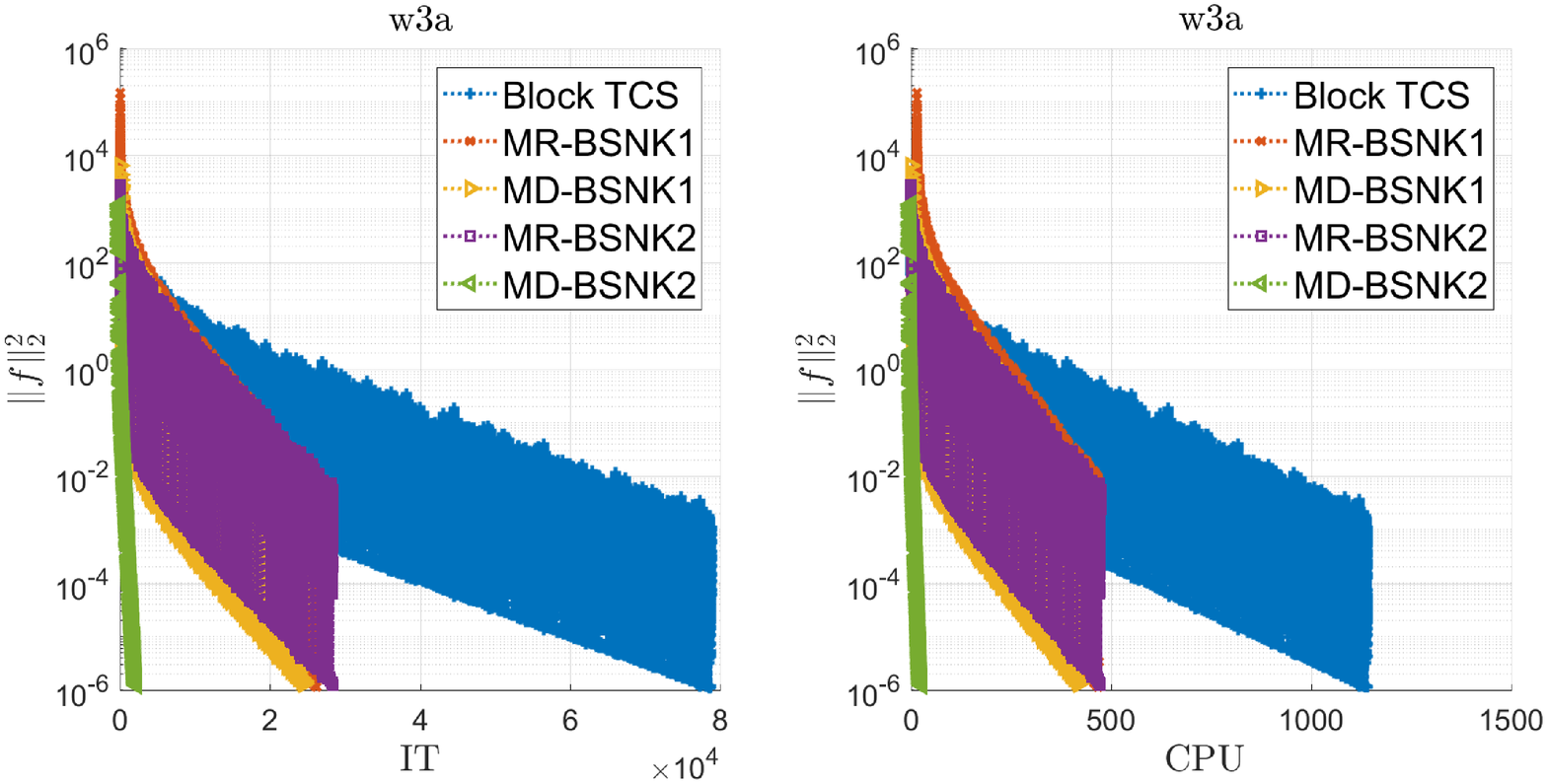}
\includegraphics [height=3.5cm,width=7cm ]{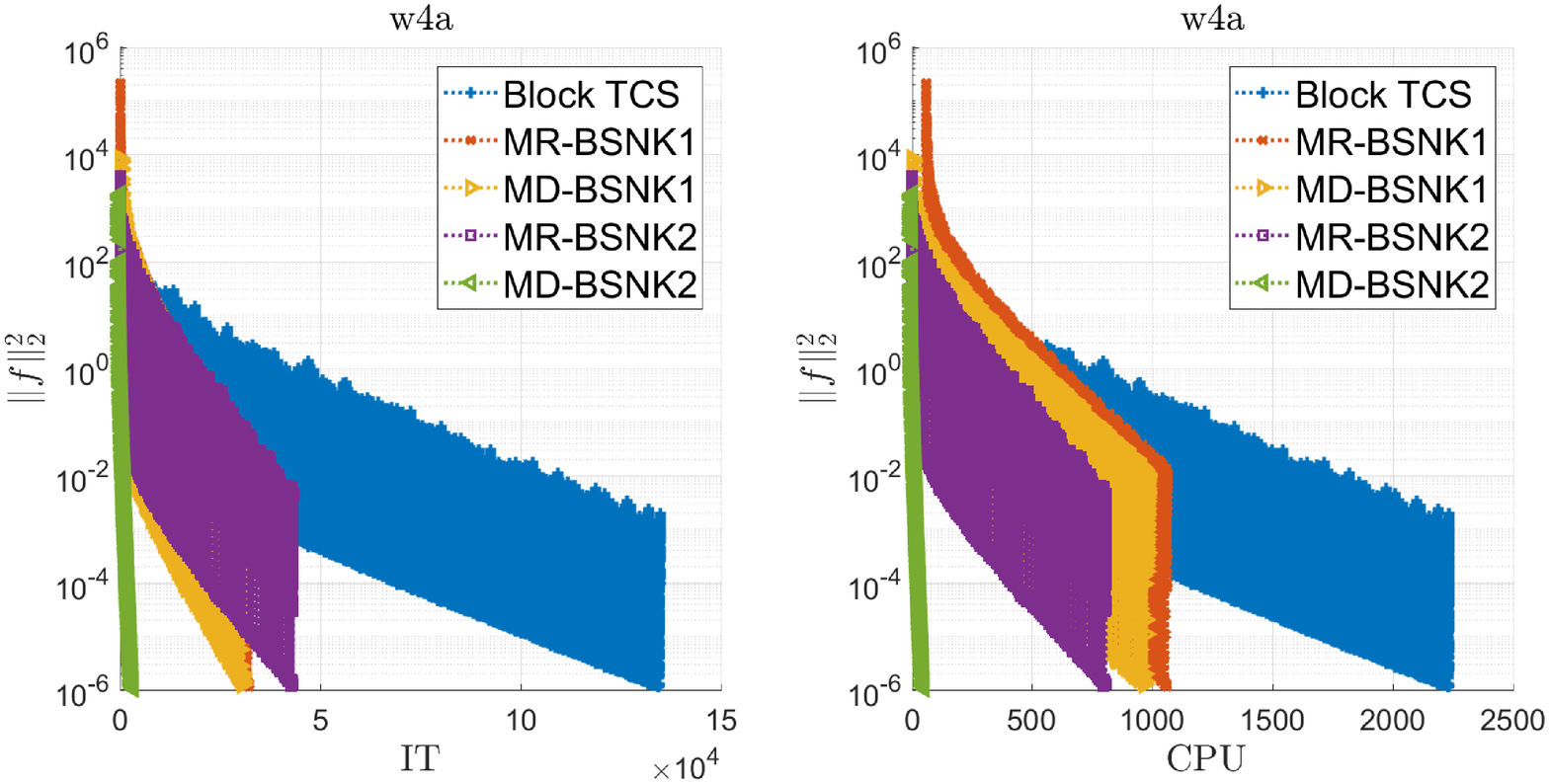}
 \end{center}
\caption{$\|f\|_2^2$ versus IT and CPU for the Block TCS, MR-BSNK1, MD-BSNK1, MR-BSNK2 and MD-BSNK2 methods with $\beta=35$ and $\nu=\tau_p=150$.}\label{fig_multi_samples}
\end{figure}

\Cref{fig_single_sample_IT,fig_single_sample_CPU,fig_multi_samples} show that our methods outperform the corresponding SNR method in terms of computing time and iteration numbers for all datasets, whether it is in the form of single-sample iteration or multi-sample iteration.

Overall, these results confirm that our methods, i.e., the MR-SNK, MD-SNK, MR-BSNK1, MD-BSNK1, MR-BSNK2 and MD-BSNK2 methods are efficient for a wide variety of problems and datasets. This fact also further reflects the superiority and developability of greedy randomized sampling.

\section{Concluding remarks}
\label{sec:conclusions}

This paper mainly proposes six greedy randomized sampling nonlinear Kaczmarz methods for solving nonlinear problems. In theory, we analyze these methods in detail and rigorously compare the size of their convergence factors. 
Theoretical results show that the convergence factors of the MR-SNK and MD-SNK methods are smaller than those of the NRK and NURK methods. More interesting, for the NURK method, we obtain tighter upper bounds. Numerical results also show that our new methods are competitive. The above findings imply that these greedy strategies open up new venues for nonlinear problems.

Considering the superiority of greed, other greedy strategies can be further explored. In addition, similar to \cite{gower2015randomized,gower2021adaptive}, by introducing an matrix $B$, more general updating formula can be deduced from the following relation:
\begin{align}
x_{k+1}=\mathop{\text{argmin}}\limits_{x\in \mathbb{R}^{n}}\|x-x_k\|_B^2,\quad s.t. \quad f_i(x_k)+\nabla f_i(x_k)^T(x-x_k)=0.\notag
\end{align}
So, we can further extend the greedy strategies discussed in this paper to the above iteration formula to obtain more efficient algorithms.



\bibliography{mybibfile}

\end{document}